\documentclass[11pt]{amsart}
\usepackage{amssymb,amsmath,amsthm,mathtools}
\usepackage{a4wide}
\usepackage{graphicx}
\usepackage[utf8]{inputenc} 
\usepackage{a4wide}

\usepackage{microtype}
\usepackage{hyperref}
\usepackage{amsfonts} 
\usepackage{latexsym}
\usepackage[font=small,format=hang,labelfont={sf,bf}]{caption}
\usepackage{epsfig}
\usepackage{subfig}
\usepackage{url}
\usepackage{varioref}
\usepackage{bm}
\mathtoolsset{showonlyrefs}
\usepackage{color}
\usepackage{mathrsfs}

\usepackage{latexsym}
\usepackage{bbm}
\usepackage{hyperref}

\usepackage{mathtools}
\usepackage{yhmath}

\allowdisplaybreaks[1]

\numberwithin{equation}{section}

\theoremstyle{plain}
\begingroup
\newtheorem{theorem}{Theorem}[section]
\newtheorem{lemma}[theorem]{Lemma}
\newtheorem{proposition}[theorem]{Proposition}
\newtheorem{corollary}[theorem]{Corollary}
\endgroup

\theoremstyle{definition}
\begingroup
\newtheorem{definition}[theorem]{Definition}
\newtheorem{remark}[theorem]{Remark}

\endgroup

\newcommand{\N}{\mathbb{N}}

\newcommand{\R}{\mathbb{R}}
\newcommand{\C}{\mathbb{C}}
\newcommand{\K}{\mathbb{K}}

\newcommand{\Rd}{\R^d}

\newcommand{\e}{\varepsilon}
\renewcommand{\div}{\mathrm{div}}
\newcommand{\sign}{\mathrm{sgn}\;}

\newcommand{\supp}{\mathrm{supp}\;}
\newcommand{\ud}{\,\mathrm{d}}
\def\H{\mathcal{H}}

\newcommand\quotient[2]{
	\mathchoice
	{
		\text{\raise1ex\hbox{$#1$}\Big/\lower1ex\hbox{$#2$}}%
	}
	{
		#1\,/\,#2
	}
	{
		#1\,/\,#2
	}
	{
		#1\,/\,#2
	}
}

\newcommand{\LamCd}{\Lambda^1_{\C}}
\newcommand{\LamRd}{\Lambda^1_{\R}}

\definecolor{ddorange}{rgb}{1,0.5,0}
\definecolor{ddcyan}{rgb}{0,0.2,1.0}

\newcommand{\hd}{\mathcal{H}^{d-1}}
\newcommand{\dv}{\mathrm{div}\;}

\newcommand{\Curr}{\mathcal{D}_1}
\newcommand{\LK}{L_{\mathscr{F}[K_s]}^2(\R^d, \C^d )}
\newcommand{\LKU}{L_{\mathscr{F}[K_s]}^2(U, \C^d )}
\newcommand{\Cs}{\mathcal{C}_s}

\title[A notion of $s$-fractional mass for $1$-currents in higher codimension] {A notion of $s$-fractional mass for $1$-currents in higher codimension}

\author[M. Cicalese]
{M. Cicalese}
\address[Marco Cicalese]{Zentrum Mathematik - M7, Technische Universitat M\"unchen, Boltzmannstrasse 3, 85748 Garching, Germany	}
\email[M. Cicalese]{cicalese@.ma.tum.de}

\author[T. Heilmann]
{T. Heilmann}
\address[Tim Heilmann]{Zentrum Mathematik - M7, Technische Universitat M\"unchen, Boltzmannstrasse 3, 85748 Garching, Germany	}
\email[T. Heilmann]{heilmant@ma.tum.de}

\author[A. Kubin]
{A. kubin}
\address[Andrea Kubin]{
	Zentrum Mathematik - M7, Technische Universitat M\"unchen, Boltzmannstrasse 3, 85748 Garching, Germany	
}
\email[A. Kubin]{andrea.kubin@tum.de}

\author[F. Onoue]
{F. Onoue}
\address[Fumihiko Onoue]{
	Zentrum Mathematik - M7, Technische Universitat M\"unchen, Boltzmannstrasse 3, 85748 Garching, Germany	
}
\email[F. Onoue]{fumihiko.onoue@tum.de}

\author[M. Ponsiglione]
{M. Ponsiglione}
\address[Marcello Ponsiglione]{Dipartimento di Matematica ``Guido Castelnuovo'', Sapienza Universit\`a di Roma, Piazzale Aldo Moro 2, I-00185 Roma, Italy
}
\email[M. Ponsiglione]{ponsigli@mat.uniroma1.it}

\begin{document}
	\maketitle
		\begin{abstract}
In this paper we propose a notion of $s$-fractional mass for $1$-currents in $\R^d$. Such a notion generalizes the notion of $s$-fractional perimeters for sets in the plane  to  higher codimension one-dimensional singularities. Remarkably, the limit as $s\to 1$ of the $s$-fractional mass gives back the classical notion of length for regular enough curves in $\R^d$. 

We prove a lower semi-continuity and compactness result for sequences of $1$-currents with uniformly bounded fractional mass and support. Moreover, we prove the density  
of weighted polygonal, closed and  compact oriented curves in the class of  divergence-free 1-currents with compact support and finite fractional mass. 

Finally, we discuss some possible applications of our notion of fractional mass to build up purely geometrical approaches to the variational modeling of dislocation lines in crystals and to vortex filaments in superconductivity. 
	\end{abstract}
	\setcounter{tocdepth}{1} 
	\tableofcontents

		\section*{Introduction}
	Nonlocal energy functionals describing long-range interactions consistently garner great interest within the mathematical community.
	When concentrated on interfaces or on more general singular sets, nonlocal energy functionals have a geometrical flavor. Nonlocal versions of the classical perimeter were introduced by Visintin in \cite{Visintin91} to model long range dissipation phenomena. These became much more popular when Caffarelli, Roquejoffre, and Savin introduced  in \cite{CRS10} the notion of $s$-fractional perimeters
	$P_s$, defined on any measurable set $E \subset \mathbb{R}^d$    as the squared fractional Sobolev semi-norm of its characteristic function, i.e., by 
	\begin{equation}\label{defFractionalPeri}
		P_s(E) = \int_{E}\int_{\R^d \setminus E} \frac{1}{|x-y|^{d+s}} \,dx\,dy \, .
	\end{equation}
	
	Our aim  is to introduce and analyze a class of nonlocal energy functionals defined on  oriented curves in $\R^d$ (namely on  $1$-currents), generalizing the notion of fractional perimeters of sets in $\R^2$ to higher codimension and to not necessarily integer rectifiable $1$-currents.  The proposed notion of fractional mass (length)  takes into account the cancellation effects hidden in the definition of fractional perimeters.  It is well known that the fractional perimeter $P_s$ does not control the local perimeter, and in fact there are sets with infinite perimeter but finite fractional perimeter;   for these sets with  finely oscillating boundaries, what makes  the fractional perimeter  finite are cancellations  that occur for  pieces of the boundary near to each other but with opposite orientations: Their  contributions   cancel each other rather than summing up their mass.  This suggests to introduce a notion of fractional mass where the nonlocal interactions between pairs of points $x,\, y $  are ruled  by a kernel depending on the distance $|x-y| $   with a (changing sign) pre-factor given  by  the scalar product between the tangent vectors to the curve at the points $x$ and $y$, respectively.
	Precisely, the $s$-fractional mass of a smooth curve $\gamma$ is defined by
	\begin{equation}\label{deforga}
		\mathrm{M}_s (\gamma) = \int_{\gamma}\int_{\gamma} \frac{ \tau(x)\cdot \tau(y)}{|x-y|^s}  \, d \H^1\, d\H^1\
	\end{equation}
	where $\tau(z)$ is the tangent unit vector of $\gamma$  at $z\in\gamma$ and $\H^1$ denotes the standard one dimensional Hausdorff measure.
	
	This notion of fractional mass not only mimics fractional perimeters, but it really extends them; in fact, on the one hand, integrating by parts \eqref{defFractionalPeri} one can check that the $s$-fractional perimeter of a set in $\R^2$ coincides (up to a prefactor)
	with the fractional mass 	 $\mathrm{M}_s$ of its (oriented) boundary; on the other hand,   $\mathrm{M}_s$ makes sense also for (not necessarily integer rectifiable) $1$-currents  in higher codimension. To understand this and identify the natural domain of $\mathrm{M}_s$  it is convenient to think in Fourier space. 
	Given a $1$-form $\gamma$  with compact support and finite mass, we consider the corresponding vector valued measure $\mu_\gamma$ and set (according with \eqref{deforga})
	\begin{equation}\label{fractionalPeriCurve2dAsFourierTransform}
		\mathrm M_s(\gamma): = {c_{d,s}} \int_{\mathbb{R}^2} |\mathscr{F}[\mu_{\gamma}]|^2 \, \frac{dx}{|x|^{d-s}}
	\end{equation} 
	where $c_{d,s}$ is a positive constant depending only on $d$ and $s$ and $\mathscr{F}[\mu_{\gamma}]$ denotes the Fourier transform of $\mu_\gamma$.
	As a consequence of this representation, $1$-forms with bounded $s$-fractional mass correspond to measurable vector fields in Fourier 
	space with a suitable (depending on $s$) weighted $L^2$ norm  finite.  This observation leads to natural lower semicontinuity (see Proposition \ref{lose}) and compactness properties in the dual norm of regular enough $1$-forms with compact support  (Theorem \ref{compthm}).
	
	In Theorem \ref{teoremaapprosi} we prove the density of weighted polygonal, closed and  compact oriented curves in the class of  divergence-free 1-currents with compact support and finite fractional mass. This result can be regarded as a nonlocal extension of the result by Marchese and Wirth in \cite{MaWi19} and Fortuna and Garroni in \cite[Theorem 3.2]{FoGa23},   related to the so-called Smirnov's decomposition theorem  \cite{Smirnov93} concerning the decomposition of  divergence-free vector fields through  elemental solenoids.  
	
	Remarkably, as $s\to 1$, the $s$-fractional mass, suitably scaled,  converges to the classical local length (see Proposition \ref{limitFracMassSto1SmoothCurves}). This is consistent with the asymptotic analysis developed for $s$-fractional perimeters \cite{Da2002,ADPM11, Po04} and Gagliardo seminorms \cite{BBM01, BBM02, CDKNP23} as $s\to 1$. Clearly, replacing $s$ with $1$ in the definition of $\mathrm M_s$ \eqref{deforga} would give back an infinite quantity for all non trivial given $\gamma$. On the other hand, such a formally  infinite energy functional seems to be relevant in many physical systems; for instance, suitable anisotropic and material dependent variants of $\mathrm M_1$ give back the self energy of dislocation lines in metals (\cite[formula 4.44]{HL82}).  To cut off the infinite core energy, a classical approach, referred to as the core radius  approach, consists in introducing a length scale $\varepsilon>0$ representing the lattice spacing and in removing the elastic energy in a  $\varepsilon$-neighborhood of the dislocation line.  The resulting energy blows up logarithmically as $\varepsilon$ vanishes. This resembles the behavior of the Ginzburg Landau energy for vortex filaments in superconductivity \cite{BBH94}. In both models, the energy is diffused in a neighborhood of the line singularity, and concentrates on it as $\e\to 0$. It seems then  interesting  to introduce a  purely geometric counterpart of the $\varepsilon$-regularized dislocations and Ginzburg Landau energy functionals,  defined directly on the singular line $\gamma$ as
	\begin{equation}\label{deforga2}
		\mathrm{M}_1^\varepsilon (\gamma) = \int_{\gamma}\int_{\gamma}  \tau(x)\cdot \tau(y)  K_\e(|x-y|)  \, d \H^1\, d\H^1\ \,,
	\end{equation}
where 		$K_\e$ takes the form $K_\e(z):=\frac{1}{\max\{|z|, \varepsilon\}}$ (or better, suitable  variants of it that are strictly monotone also around zero). 
	In fact, the identification and proposal of this functional is one of the motivations behind this work. Before analyzing such a functional, it seemed appropriate to consider its counterparts $\mathrm M_s$ with finite energy, that is, with 
	$s \in (0,1)$. The analysis of the regularized critical limit case $\mathrm{M}_1^\varepsilon$ goes beyond the scopes of this paper, but in our opinion deserves further investigation. 
	
	Another natural extension of our theory consists in generalizing the $s$-fractional mass for $1$-currents in $\R^d$ to the general case of  $k$-currents for any $1\le k\le d$.  
	Recently, several authors have tried to extend the notion of the fractional length or area to more general (smooth) manifolds even with boundary. For instance, Paroni, Podio-Guidugli, and Seguin in \cite{PPgS18} introduced the notion of fractional area for any smooth hypersurface with boundary and computed the Euler-Lagrange equation associated with such a functional (see also \cite{PPgS22} for the notion of fractional curvature tensor). Moreover, Seguin in \cite{Seguin20} introduced a notion of fractional length for (not necessarily closed)  smooth curves in $\mathbb{R}^d$ with $d \geq 2$.
	We also remark that O'Hara introduced an energy functional for smooth curves in $\mathbb{R}^3$, which is known as the O'Hara energy, in a series of papers \cite{OHara91, OHara92, OHara92No2, OHara94}. Such an energy functional represents  a sort of ``electrostatic potential energy'' for curves in $\mathbb{R}^3$ and  it takes its minimum when a curve is a circle and it blows up when a curve has a self-intersection. In some respects our proposed $s$-fractional mass represents a natural counterpart of the O'Hara energy taking into account cancellation effects for pieces of curves with opposite orientations, that attract rather than repel each other. While a thin and elongated curve, driven by the O'Hara energy, would luckily converge to a flat disc, the corresponding evolution driven by $\mathrm M_s$ should attract such lines with opposite orientations, leading  to their annihilation. This dynamics mirrors exactly the dynamics observed in dislocation and vortex filaments. Our proposed fractional length represents a first attempt to establish a sound variational model for these phenomena within a purely geometric framework.


	
\vskip15pt

	{\bf{Notation of the paper}}
	 We denote by $\R^d$ the $d$-dimensional Euclidean space and by $e_1,\dots,e_d$ its standard orthonormal basis. We then denote by $\lambda_1,\dots, \lambda_d$ the standard basis of the dual space $\Lambda^1_{\R}:=\Lambda^1 (\R^d)$ of $\R^d$. We denote by $\Lambda^1_{\C} \coloneqq \Lambda^1(\C^d) (\cong \Lambda^1(\R^{2d}))$ the dual space of $\C^d$. Notice that $\Lambda^1_{\R}$ is the $\R$-subspace of $\LamCd$.
	
	We denote by $\omega_d$ the volume of the unit open ball $B_1(0)$ of $\mathbb{R}^{d}$. Given any Lebesgue measurable set $E\subset \mathbb{R}^d$, we denote by $|E|$ or euivalently by $\mathcal L^d(E)$ its Lebesgue measure. We denote by $\mathcal{M}_{b}(\mathbb{R}^{d})$ the space of (non negative) finite Radon measures in $\mathbb{R}^{d}$. The Dirac delta measure centered in $x$ is denoted by $\delta_{x} $. 
	
	We denote by $\mathcal{D}^1$ the space of compactly supported, smooth 1-forms $ C_c^{\infty}(\R^d,\, \Lambda^1_{\R})$. The dual space of $\mathcal{D}^1$ is the space of $1$-currents and it is denoted by $ \Curr$. For any $\gamma \in \Curr$, we define its support by $\mathbb{R}^d \setminus \cup W$ where the union is over all open sets $W \subset \mathbb{R}^d$ such that $\gamma(\omega)=0$ for any $\omega \in \mathcal{D}^1$ with $\supp \omega \subset W$. Given $\gamma \in\mathcal{D}_1$, we define the boundary of $\gamma$ by the 0-current such that $ \partial \gamma(\omega):= \gamma (d \omega)$ for all $ \omega \in C_c^{\infty}(\R^d,\R)$. If $(X,\| \cdot\|)$ is a normed vector space and $Y \subset X$ and $\K$ is a topological field, we denote with $ C((Y,\| \cdot\|),\K)$ the space of all continuous function from $(Y,\| \cdot\|)$ to $\K$ endowed with the induced topology of $(X,\| \cdot\|)$. If $z \in \C$, we denote by $\overline{z}^*$ its complex conjugate.
	We denote by $C(*,\cdots,*)$ a positive constant that depends on finite factors $*,\cdots,*$; this constant can vary from line to line. 
	
	\section{A notion of $s$-fractional mass for $1$-current and its properties}\label{Sec1}
	The Fourier transform of an integrable function $f \colon \R^d \rightarrow \R $ is defined as
	\begin{equation}\label{Fourtrasform}
		\mathscr{F}[f](\xi):= \int_{\R^d} f(x)e^{-ix \cdot \xi}\,dx
	\end{equation}
	and its inverse $\mathscr{F}^{-1}$ is defined as
	\begin{equation*}
		\mathscr{F}^{-1}[f](x) \coloneqq \frac{1}{(2\pi)^d} \int_{\R^d} f(\xi) \, e^{i\,\xi \cdot x}\,d\xi.
	\end{equation*}
	Let $s\in(0,1)$ be fixed. For $0<\alpha <d$ we consider the kernel $K_\alpha:\R^d \to (0,+\infty)$ defined as
	$$
	K_\alpha(x):= \frac{1}{\vert x \vert^{\alpha}}\, 
	$$
	
	We highlight that  the Fourier transform $\mathscr{F}[K_\alpha]$ is 
	\begin{equation}\label{fouritra}
		\mathscr{F}[K_\alpha](\xi)=c(\alpha,d)K_{d-\alpha}(\vert \xi \vert)\,,
	\end{equation}
	where \begin{equation}\label{20122023matt1}
		c(\alpha,d):=2^{d-\alpha}\pi^{\frac d 2}\frac{\Gamma(\frac{d-\alpha}{2})}{\Gamma(\frac{\alpha}{2})}
	\end{equation}
	and $\Gamma:(0,+\infty) \rightarrow (0,+\infty)$ is the Euler Gamma function, namely $\Gamma(\beta):=\int_{0}^{+\infty}t^{\beta-1}e^{-t}\ud t$\,, for $\beta>0$\,.
	Notice in particular that $\mathscr{F}[K_{\alpha}]>0$. 
	
	We define the family $\Gamma$ of the finite union of simple $C^1$ curves (not necessarily closed), namely
	\begin{equation}\label{defGamma}
		\Gamma:= \left\{\bigcup_{i=1}^N \gamma_i \colon \, N\in \N, \gamma_i\in C^1([a_i,b_i], \R^d) \text{ for } a_i < b_i\right\}\,.
	\end{equation}
	The $s$-fractional mass of a curve $\gamma \in \Gamma$ is defined as
	\begin{align}\label{defps}
\mathrm{M}_s (\gamma) &\coloneqq \sum_{i,j=1}^N \int_{\gamma_i}\int_{\gamma_j} K_s(x-y) 
		  \tau_i(x) \cdot \tau_j(y)   \, d \H^1\, d\H^1\ \nonumber\\
		  &= \sum_{i,j=1}^N \int_{\gamma_i}\int_{\gamma_j} \frac{ \tau_i(x)\cdot \tau_j(y)}{|x-y|^s}  \, d \H^1\, d\H^1\
	\end{align}
	where $\tau_i$ denotes the unit tangent vector of $\gamma_i$ for every $i \in \{1,\,2,\cdots,N\}$. 

	The next proposition states that in dimension 2, the above notion of the $s$-fractional mass for a closed, simple, smooth curve $\gamma$ coincides (up to.a multiplicative constant) with that of fractional perimeter of a set $E$ such that $\partial E=\gamma$ .
	\begin{proposition}\label{propositionFracPerivsFracMass}
		Let $ E \subset \R^2 $ be an open bounded set with $C^1$ boundary. 
		For all $ s \in (0,1)$
		$$ \mathrm{P}_s(E)=\frac{1}{s^2} \mathrm{M}_s(\partial E)$$
		where we identify  $\partial E$ as an element of  $\Gamma$ in the obvious way.
	\end{proposition}
	\begin{proof}
		
		For all $x,y \in \R^2 \text{ and } \nu \in \mathbb{S}^1$ we have
		\begin{equation}\label{01072023matt1}
			\div_y \left(\frac{y-x}{\vert y-x \vert^{2+s}}\right)=-s \frac{1}{\vert y-x \vert^{2+s}}, \quad \div_x \left(\frac{\nu}{\vert x-y \vert^s}\right)= s \frac{(y-x)\cdot \nu}{\vert y-x\vert^{2+s}} .
		\end{equation}
		Therefore applying the divergence theorem and the above formula we have
		\begin{equation}
			\begin{split}
				\mathrm{P}_s(E)=\int_{E} \int_{E^c} \frac{1}{\vert x- y \vert^{2+s}} &\, dx dy=\frac{1}{s}
 \int_{E} \int_{\partial E} \frac{(y-x)\cdot \nu_{E}(y)}{\vert y-x\vert^{2+s}}\, d \mathcal{H}^1(y)\, dx \\
				&=\frac{1}{s^2} \int_{\partial E } \int_{\partial E} \frac{\nu_E(x) \cdot \nu_E(y)}{\vert y-x
 \vert^s} \, d \mathcal{H}^1(x) d \mathcal{H}^1(y)\\
				&=\frac{1}{s^2} \int_{\partial E } \int_{\partial E} \frac{\tau_E(x) \cdot\tau_E(y)}{\vert y-x
 \vert^s} \, d \mathcal{H}^1(x) d \mathcal{H}^1(y)= \frac{1}{s^2} \mathrm{M}_s(\partial E).
			\end{split}
		\end{equation}
	\end{proof}

%
	
	We now compute the first variation of the $s$-fractional mass for curves in $\Gamma$.
	\begin{proposition}[First variation of $s$-fractional mass]\label{propositionFirstVariFracMass}
	Let $\gamma : [0,\,\ell] \to \mathbb{R}^d$ be an arc-length parametrization of a simple $C^{1}$-curve (not necessarily closed). Given any $h \in C^{1}([0,\,\ell], \mathbb{R}^d)$ with $h(0)=h(\ell)=0$, we set $\gamma_t \coloneqq \gamma + t h$ for all $0 \leq t < 1$. Then, for sufficiently small $t>0$, $\gamma_t$ is also a simple $C^{1}$-curve and
	\begin{align}
		\qquad &\frac{d}{dt} \Big\lfloor_{t=0} \mathrm M_s(\gamma_t) \nonumber\\
		&= 2 \iint_{[0,\,\ell]^2} \left[ \frac{\dot{h}(u) \cdot \dot{\gamma}(v)}{|\gamma(u) - \gamma(v)|^s} 
 		- \frac{\frac{s}{2}\,\dot{\gamma}(u) \cdot \dot{\gamma}(v)}{|\gamma(u) - \gamma(v)|^{2+s}} (\gamma(u) - \gamma(v)) \cdot (h(u)-h(v)) \right] \,du\,dv \label{firstVariFracMassComputation} \\
		&=s\int_{0}^{\ell} h(u) \cdot \left\{ \int_{0}^{\ell} \frac{\left((\gamma(u)-\gamma(v))\cdot \dot{\gamma}(u)\right)\,\dot{\gamma}(v) - \left(\dot{\gamma}(u)\cdot \dot{\gamma}(v)\right)\, (\gamma(u)-\gamma(v))}{\vert \gamma(u)-\gamma(v)\vert^{2+s}} dv \right\} du. \label{25122023matt2} 
	\end{align}
	\end{proposition}
	
	\begin{proof}
	We note that $\gamma_t$ is a simple $C^{1}$-curve for $t>0$ small enough.
	Let $0<\varepsilon<1$ and for  any simple $C^1$-curve $\eta$  set 
	\begin{equation*}
		\mathrm M_s^{\varepsilon}(\eta) \coloneqq \iint_{\{|u-v| \geq \varepsilon\}} \frac{\dot{\eta}(u) \cdot \dot{\eta}(v)}{|\eta(u)-\eta(v)|^s} \,du\,dv. 
	\end{equation*}
	Then we have $\lim_{\varepsilon \downarrow 0} \mathrm M_s^{\varepsilon}(\eta) = \mathrm M_s(\eta)$.

	We first compute the following derivative:
%
		\begin{align}
			\frac{d}{dt}\left(\frac{\dot{\gamma_t}(u) \cdot \dot{\gamma_t}(v)}{|\gamma_t(u) - \gamma_t(v)|^s} \right)  &= \frac{\dot{\gamma_t}(u) \cdot \dot{h}(v)}{|\gamma_t(u) - \gamma_t(v)|^s} + \frac{\dot{h}(u) \cdot \dot{\gamma_t}(v)}{|\gamma_t(u) - \gamma_t(v)|^s} \nonumber\\
			&\qquad  -  \frac{s(\dot{\gamma_t}(u) \cdot \dot{\gamma_t}(v))}{|\gamma_t(u) - \gamma_t(v)|^{1+s}} \frac{(\gamma_t(u) - \gamma_t(v)) \cdot (h(u) - h(v))}{|\gamma_t(u) - \gamma_t(v)|} \nonumber\\
		\end{align}
		for any $u,\, v \in [0,\,\ell]$ with $u \neq v$. Notice that $\gamma(u) \neq \gamma(v)$ for $u \neq v$ since $\gamma$ is simple. Thus we have
		\begin{align}\label{compuFirstVariationFracMass}
			\frac{d}{dt} \mathrm M_s^{\varepsilon}(\gamma_t)  &=  \iint_{\{|u-v|\geq \varepsilon\}} \frac{d}{dt} \left( \frac{\dot{\gamma_t}(u) \cdot \dot{\gamma_t}(v)}{|\gamma_t(u) - \gamma_t(v)|^s} \right) \,du\,dv \nonumber\\
			&= 2 \iint_{\{|u-v|\geq \varepsilon\}} \frac{\dot{h}(u) \cdot \dot{\gamma_t}(v)}{|\gamma_t(u) - \gamma_t(v)|^s} \,du\,dv \nonumber\\
			&\qquad - s \iint_{\{|u-v|\geq \varepsilon\}} \frac{\dot{\gamma_t}(u) \cdot \dot{\gamma_t}(v)}{|\gamma_t(u) - \gamma_t(v)|^{2+s}} (\gamma_t(u) - \gamma_t(v)) \cdot (h(u) - h(v)) \,du\,dv  
		\end{align}
		for any $\varepsilon > 0$.

Since $\gamma$ is simple and of class $C^1$ and $h$ also belongs to $C^1$, for $t$ small enough it holds that
$$
|\gamma_t(u) -\gamma_t(v)|\ge c |u-v| \qquad \text{ for all } u,\, v\in  [0,\,\ell] \, ,
$$
for some constant $c>0$ independent of $t$. 
Using again  that $\gamma, \, h\in C^1$, both integrands in  \eqref{compuFirstVariationFracMass} are bounded by $c\,|u-v|^{-s}$, which is integrable on $[0,\,\ell]^2$.
%
		Therefore, $\frac{d}{dt} \mathrm M_s^{\varepsilon}(\gamma_t)$ is uniformly bounded and converges uniformly to the integral in \eqref{firstVariFracMassComputation} as $\varepsilon \to 0$. 

		The formula \eqref{25122023matt2} follows by  \eqref{firstVariFracMassComputation}  integrating by parts. 
	\end{proof}
	
%
%
%
	
	Motivated by Proposition \ref{propositionFirstVariFracMass}, we can define the so-called ``$s$-fractional curvature'' of simple smooth curves as follows.
	\begin{definition}[$s$-fractional curvature]\label{defFractionalCurvature}
		Let $s \in (0,1)$ and $\gamma : [0,\,\ell] \rightarrow \R^d$ be an arc-length parametrization of a simple $C^2$-curve (not necessarily closed). For every $u \in (0,\,\ell)$, we define $s$-fractional curvature of $\gamma$ at $ u$ as
		\begin{equation}\label{curvaturefraz}
			k_{s}(u,\gamma):=s \int_{0}^{\ell}\frac{\left((\gamma(u)-\gamma(v))\cdot \dot{\gamma}(u)\right) \, \dot{\gamma}(v)-\left(\dot{\gamma}(u)\cdot \dot{\gamma}(v)\right)\,(\gamma(u)-\gamma(v))}{\vert \gamma(u)-\gamma(v)\vert^{2+s}} dv.
		\end{equation} 
	\end{definition}
	
	Now we aim at introducing the notion of $s$-fractional mass for 1-currents. This requires some preparation.
	First of all, any curve $\gamma \in \Gamma$ can be seen as a continuous linear functional defined as 
	\begin{equation}
		\gamma \colon \mathcal{D}^1 \rightarrow \R, \quad \gamma (\omega):= \int_{\gamma} \langle \omega (x), \tau(x) \rangle \, d \mathcal{H}^1(x)
	\end{equation}
	where $ \tau(x)$ is the tangent unit vector to $\gamma$ at $x$.  	 
	For $\e>0$ we let $ \rho_{\varepsilon} \in C^{\infty}_c(B_1)$ be standard Friedrich mollifiers, i.e., smooth functions with compact support such that $\int_{\R^d}\rho_{\varepsilon}\,dx = 1$ and $\lim_{\varepsilon \to 0} \rho_{\varepsilon} = \delta_0$ in the sense of distributions. Then, for any $\gamma\in\Curr$ we can define its convolution $\rho_\varepsilon * \gamma : \mathcal{D}^1 \to \R$ by  
\begin{equation}\label{convgamma}
 \gamma  * \rho_\varepsilon (\omega)  :=  \gamma(  \rho_\varepsilon * \omega)
\end{equation}
for any $\omega \in \mathcal{D}^1$.

We define the function space $C_0^\infty(\R^d, \LamCd)$ by the closure of $C^\infty_c(\R^d,\LamCd)$ with respect to the family of seminorms $\{\|\cdot\|_{\beta}\}_{\beta \in \N^d}$, where $ \| f\|_{\beta}:= \sup_{x \in \R^d} \vert D^{\beta}f(x) \vert$ and $ D^\beta f(x)= \partial_1^{\beta_1}\partial_2^{\beta_2}\cdots\partial_d^{\beta_d} f(x)$. In other words, we say $f \in C^\infty_0(\R^d, \LamCd)$ if there exists $ \{f_n\}_{n \in \N} \subset C^\infty_c(\R^d,\LamCd)$ such that, for all $ \beta \in \N^d$, $ \| f_n -f\|_{\beta} \rightarrow 0$ as $ n \rightarrow +\infty$. We remember that  $C_0^\infty(\R^d,\LamCd)$ is a Fréchet space.
 Moreover, for any $\eta \in C_0^{\infty}(\R^d, \LamCd)^*$ (the continuous dual space of $C^{\infty}_0(\R^d, \LamCd)$), we define 
\begin{equation}\label{conv123}
\eta * \rho_\varepsilon(\omega)	:=  \eta(  \rho_\varepsilon * \omega)
\end{equation}
for any $\omega \in C_0^{\infty}(\R^d, \LamCd)$.

We define the set of the $ \LamCd$-valued Schwartz functions as
\begin{equation}
	\mathcal{S}(\R^d,\LamCd):= \left\{f \in C^{\infty}(\R^d,\LamCd) \colon \| f \|_{N,\beta} < \infty \, \forall N\in \N, \beta \in \N^d\right\}
\end{equation}
where $ \| f\|_{N,\beta}:= \sup_{x \in \R^d}  (1+ \vert x \vert)^N \vert D^\beta f(x) \vert$ for any given $ N \in \N$ and $ \beta= (\beta_1,\dots,\beta_d) \in \N^d$. We denote by $\mathcal{S}(\R^d,\LamCd)^*$ the dual space of $\mathcal{S}(\R^d,\LamCd)$. We recall that the collection of seminorms $\{\|\cdot\|_{N,\beta}\}_{N,\beta}$ induces a topology on $\mathcal{S}(\R^d, \LamCd)$, namely, $f_n \to f$ in $\mathcal{S}$ as $n \to +\infty$ if $\|f_n-f\|_{N,\beta} \to 0$ for any $N \in \N$ and $ \beta \in \N^d$. We remember that  $\mathcal{S}(\R^d,\LamCd)$ is a Fr\'echet space.

Let $ U \subset \R^d$ be an open set and let $k \in \N,\, \alpha \in [0,1]$. We define $C_0^{k,\alpha}(U,\LamCd)$ as the closure of $C^{k,\alpha}_c(U,\LamCd)$ with respect to $\|\cdot\|_{C^{k,\alpha}}$, i.e., we say $f \in C^{k,\alpha}_0(U,\LamCd)$ if there exists a sequence $ \{f_n\}_{n \in \N} \subset C_c^{k,\alpha}(U,\LamCd)$ such that  $ \| f_n -f \|_{C^{k,\alpha}(U,\LamCd)} \rightarrow 0$ as $ n \rightarrow + \infty$. We recall  that $C_0^{k,\alpha}(U,\LamCd)$ is a Banach space. Given a closed set $C \subset \R^d$  we define $ C^{k,\alpha}(C,\LamCd)$ as the intersection of all $  C^{k,\alpha}(U,\LamCd)$ where $ C \subset U$ with $U $ open. \\

Now we prove that any 1-current with compact support can be  uniquely extended to a bounded linear functional on $C^{\infty}_0(\R^d, \LamCd)$. Here we consider the space of 1-forms $\mathcal{D}^1$ as the $\R$-subspace of $C^{\infty}_0(\R^d, \LamCd)$.  
	\begin{lemma}\label{lemmaexconconv}
		Let $ \gamma \in \Curr$ with compact support, then there exists a unique extension $\overline{\gamma} \in C_0^{\infty}(\R^d, \LamCd)^* \, (\subset \mathcal{S}(\R^d,\LamCd)^*)$.
		Moreover  $\overline{ \gamma} * \rho_{\varepsilon} = \overline{\gamma * \rho_{\varepsilon}}$ for all $ \varepsilon \in (0,1)$ and $ \lim_{\varepsilon \rightarrow 0} \overline{ \gamma} * \rho_{\varepsilon} = \overline{\gamma}$ in the weak star topology of $ C_0^{\infty}(\R^d,\LamCd)^*$. 
	\end{lemma}
\begin{proof}
Since $ C_0^{\infty}(\R^d,\LamCd)$ and $ \mathcal{S}(\R^d,\LamCd)$ satisfy the first axiom of countability (because they are Fr\'echet spaces), then we can define the extension only by means of converging sequences. 
	Let $K$ be a compact set that contains the support of $ \gamma$.
	Let $ \varphi \in C_c^{\infty}(\R^d,\R) $ such that $ \varphi(x)=1$ for all $ x \in K$. For all $ \omega \in   C_0^{\infty}(\R^d, \LamCd) $ we consider a sequence $ \{\omega_n\}_{n \in \N} \subset \mathcal{D}^1$ such that $ \omega_n \rightarrow \omega$ in  $ C_0^{\infty}(\R^d, \LamCd) $. We observe that the sequence $ \{\varphi \omega_n\}_{n \in \N}$ converge to $ \varphi\omega$ in $ \mathcal{D}^1$ as $ n \rightarrow +\infty$. Therefore, we have that, for all $\varepsilon>0$, there exists $\overline{n} \in \N$ such that, for all $n,m >\overline{n}$,
	\begin{equation}
		 \vert \gamma(\omega_n)- \gamma(\omega_m) \vert = \vert \gamma(\omega_n-\omega_m) \vert= \vert \gamma(\varphi\omega_n-\varphi\omega_m )\vert \leq \varepsilon
	\end{equation}
where we have used the continuity of $ \gamma $. Hence,  we can define $ \overline{\gamma}(\omega):= \lim_{n \rightarrow +\infty} \gamma(\omega_n)= \gamma(\varphi \omega)$. Moreover, it is easy to check that the definition of $ \overline{\gamma}$ does not depend on the chosen sequence and of the specific choice of $ \varphi$; 
this also shows the uniqueness of the extension and the continuity of the extension and completes the proof of the first assertion. We have that $ \overline{\gamma} \in \mathcal{S}(\R^d,\LamCd)^*$ observing that if $\{\omega_n\}_{n \in \N} \subset \mathcal{D}^1$ is such that $ \omega_n \rightarrow \omega $ in $\mathcal{S}(\R^d,\LamCd)$ then $ \omega_n \rightarrow \omega$ in  $C_0^{\infty}(\R^d,\LamCd)$. 

From \eqref{convgamma} and \eqref{conv123}, we have that, for all $ \omega \in \mathcal{D}^1 $, 
\begin{equation}
 \overline{\gamma * \rho_{\varepsilon}}(\omega)= \gamma * \rho_{\varepsilon}(\omega)= \gamma(\rho_{\varepsilon} * \omega)= \overline{\gamma}(\rho_{\varepsilon} * \omega)= \overline{\gamma}* \rho_{\varepsilon}(\omega)
\end{equation}
and by density the above identities hold true for all  $ \omega \in C_0^{\infty}(\R^d,\LamCd) $.
Finally, given $ \omega \in C_0^{\infty}(\R^d,\LamCd)$ 
we have
\begin{equation}
	\lim_{\varepsilon \rightarrow 0}\overline{\gamma}* \rho_{\varepsilon} (\omega)= \lim_{\varepsilon \rightarrow 0} \overline{\gamma}(\rho_{\varepsilon}* \omega)= \overline{\gamma}(\omega) 
\end{equation}
where we have used that $ \rho_{\varepsilon} * \omega \rightarrow \omega$, as $ \varepsilon \rightarrow 0$, in $ C_0^{\infty}(\R^{d},\LamCd)$.
\end{proof}


	 
	Next we define a function space which represents the natural domain for relaxing the definition of $s-$ fractional mass from the class $\Gamma$ to the class of 1-currents. To this purpose,  we define the weighted $L^2$ space associated with the kernel $K_{d-s}$ as follows:  
	\begin{equation}\label{L2kspace}
		\LK:= \left\{ f \colon \R^d \rightarrow \C^d \mid  \text{$f$ is measurable and $\int_{\R^d} \vert f (x)\vert^2 K_{d-s}(\vert x\vert )dx < \infty$} \right\}.
	\end{equation}
	Note that the space $\LK$ can be endowed  with the following norm
	\begin{equation*}
		\| f \|_{\LK} \coloneqq \left( c(s,d) \int_{\R^d} \frac{|f(x)|^2}{|x|^{d-s}}\,dx \right)^{\frac{1}{2}} = \left( c(s,d) \int_{\R^d} \vert f (x)\vert^2 K_{d-s}(\vert x\vert )dx \right)^{\frac{1}{2}}.
	\end{equation*} 
	
	\noindent We then define the Fourier transform of $\gamma \in \mathcal{D}_1$, which we denote by $\mathscr{F}[\gamma]$, as a $\C$-valued functional defined on $\mathcal{S}(\R^d, \LamCd)$. Precisely, we define $\mathscr{F}[\gamma]$ by
	\begin{equation}\label{ftcurr}
		  \mathscr{F}[\gamma](\omega) := \overline{\gamma}(\mathscr{F}[\omega])
	\end{equation}
	for any $\omega \in \mathcal{S}(\R^d,\LamCd)$. Here $\overline{\gamma}$ is the unique extension of $\gamma$ as in Lemma \ref{lemmaexconconv}. Notice that $\mathscr{F}[\gamma]$ is $\R$-linear and continuous on $\mathcal{S}(\R^d, \LamCd)$ because $\mathscr{F}[\gamma]$ is the composition of $\R$-linear, continuous functions $\overline{\gamma}$ and $\mathscr{F}$. 
	Moreover, we can observe that the restriction $\mathscr{F}[\gamma]\lfloor_{\mathcal{D}^1}$ is also a bounded linear functional, i.e., a 1-current. In addition, the following Plancherel-type identity for 1-currents holds for any $\omega \in \mathcal{D}^1$:
	\begin{equation}\label{fourierTransformCurrents}
		\frac{1}{(2\pi)^d}\mathscr{F}[\gamma]\left(\overline{\mathscr{F}[\omega]}^*\right) = \gamma(\overline{\omega}^*). 
	\end{equation}
	 Indeed, from the Fourier inversion theorem, we have that  $\mathscr{F}\left[\overline{\mathscr{F}[\omega]}^*\right] = (2\pi)^d \,\overline{\omega}^*$ for any $\omega \in \mathcal{D}^1$. Then we obtain
	\begin{equation}\label{17122023pom1}
		\frac{1}{(2\pi)^d}\mathscr{F}[\gamma]\left(\overline{\mathscr{F}[\omega]}^*\right) = \frac{1}{(2\pi)^d} \overline\gamma\left(\mathscr{F}\left[\overline{\mathscr{F}[\omega]}^*\right] \right) = \gamma(\overline{\omega}^*)
	\end{equation}
	for any $\omega \in \mathcal{D}^1$. Note that the Plancherel-type identity above holds true also when $\overline{\gamma}$ is the extension of $\gamma$ as in Lemma \ref{lemmaexconconv}.

	
	We are now in a position to introduce the admissible class $\mathcal{C}_s$ of 1-currents with compact support suited to define the $s$-fractional mass.
	\begin{definition}\label{setC_s}
		Let $U \subset \R^d$, we say $\gamma \in \Cs(U)$ if $\gamma$ is a 1-current with compact support contained in $U$ 
		and there exists a (unique)
		measurable function $\widehat{\gamma}:\R^d \rightarrow \mathbb{C}^d$ such that
		\begin{equation}\label{defAdmissibleClassFractionalMass1Current}
			\mathscr{F}[\gamma](\omega) = \int_{\mathbb{R}^d} \left< \omega(x), \widehat{\gamma}(x) \right>\,dx \quad \text{with $\widehat{\gamma} \in \LK$}
		\end{equation}
	for any $\omega \in \mathcal{S}(\R^d,\LamCd)$. 
		Moreover, we define the norm of $\gamma \in \mathcal{C}_s(U)$ by
		\begin{equation*}
			\|\gamma\|_{\mathcal{C}_s} \coloneqq \| \widehat{\gamma} \|_{\LK}.
		\end{equation*} 
		In what follows we simply set $\Cs \coloneqq \Cs(\R^d)$.
	\end{definition}
	
	\begin{remark}
		We observe that the quantity $\|\gamma\|_{\Cs}$ for $\gamma \in \Cs$ is actually a norm in $\Cs$. Indeed, assuming that $\|\gamma\|_{\Cs} = 0$, we have, by definition, that $\widehat{\gamma} = 0$ a.e. in $\R^d$ and thus $\mathscr{F}[\gamma] = 0$ on $\mathcal{S}(\R^d, \LamCd)$. Then, from the Plancherel-type identity, we obtain
		\begin{equation*}
			\gamma(\overline{\omega}^*) = \mathscr{F}[\gamma]\left(\overline{\mathscr{F}[\omega]}^*\right) = 0
		\end{equation*}
		for any $\omega \in \mathcal{D}^1$, which implies that $\gamma=0$ on $\mathcal{D}^1$. The triangle inequality and absolute homogeneity hold true from the definition of $\|\gamma\|_{\Cs}$.
	\end{remark}
	\begin{remark}
	We observe that	$ \left(\mathcal{C}_s,\| \cdot \|_{\Cs}\right)$ is separable.
Since $ \LK$  satisfies the second axiom of countability, its subspaces are separable. Moreover, the map 
$$ T: \left(\mathcal{C}_s,\| \cdot \|_{\Cs}\right) \rightarrow \big(\LK, \| \cdot \|_{\LK}\big), \quad  T(\gamma):= \hat{\gamma}$$ is an isometry, therefore $ \left(\mathcal{C}_s,\| \cdot \|_{\Cs}\right)$ is isometric to a topological separable space.   
	\end{remark}
\begin{lemma}
%
	If $ \gamma \in \mathcal{C}_s$, then $ \mathscr{F}[\gamma]$ can be extended into a linear and continuous functional on $C_0^{k,\alpha}(U,\LamCd)$ for any $k \in \N$, $\alpha \in (0,1)$, and bounded open set $ U \subset \R^d$.
	
\end{lemma}
\begin{proof}
	Let $U \subset \R^d$ be any bounded open set. By definition, we first observe that $ \hat{\gamma} \in \mathrm{L}^1(U, \C^d)$. Let $\omega \in C_c^{\infty}(U, \LamCd)$. Then, from the integral representation of $\mathscr{F}[\gamma]$, we have that
	\begin{equation}
		\vert 	\mathscr{F}[\gamma](\omega) \vert = \left| \int_{\R^d} \left< \omega(x), \widehat{\gamma}(x) \right>\,dx\right| \leq \| \widehat{\gamma} \|_{\mathrm{L}^1(U,\C^d)} \| \omega \|_{C^{k,\alpha}(U,\LamCd)}
	\end{equation} 
	for any $k \in \N$ and $\alpha \in (0,\,1)$. Thus, since $C_c^{\infty}(U, \LamCd)$ is dense in $C_0^{k, \alpha}(U, \LamCd)$, in the same way as in the proof of Lemma \ref{lemmaexconconv}, we can continuously extend the domain of $\mathscr{F}[\gamma]$ into $C_0^{k, \alpha}(U, \LamCd)$. This completes the proof.
\end{proof}

	In the above setting, we define the $s$-fractional mass of 1-currents in $\mathcal{C}_s$ as follows.
	\begin{definition}[$s$-fractional mass for 1-currents] \label{defFractionalMass1currents}
		For any $\gamma \in \mathcal{C}_s$, we define the $s$-fractional mass $ M_s$ of $\gamma$ by
		\begin{equation}
			M_s(\gamma) \coloneqq (2\pi)^{-d} \|\gamma\|_{\mathcal{C}_s}^2.
		\end{equation}
	\end{definition}
	Notice that, by definition, the $s$-fractional mass of $\gamma \in \mathcal{C}_s$ can be also written in the following way:
	\begin{equation*}
	M_s(\gamma) = (2\pi)^{-d} \|\widehat{\gamma}\|_{\LK}^2 = (2\pi)^{-d}
		c(s,d) \int_{\R^d} \frac{1}{\vert x \vert^{d-s}} \vert \widehat{\gamma}(x) \vert^2 \, dx
	\end{equation*}
	where $\widehat{\gamma}$ is the function in \eqref{defAdmissibleClassFractionalMass1Current} associated with $\gamma \in \mathcal{C}_s$.
	
	We observe that the $\mathcal{C}_s$-norm of smooth curves (namely, the $s$-fractional mass in Definition \ref{defFractionalMass1currents}) coincides with the $s$-fractional mass defined in \eqref{defps}.
	To see this, we first recall the following classical result, see \cite[Lemma 12.12]{Mattila99}.
	\begin{lemma}\label{lemmaMattila99}
		Let $\mu$ be a Radon measure on $\R^d$ with compact support. Then for all $0 <s<d$
		\begin{equation}
			\int_{\R^d} \int_{\R^d}  \frac{1}{\vert x -y\vert^s}\, d \mu(x) d \mu(y) 
			=  (2\pi)^{-d}c(s,d) \int_{\R^d} \frac{1}{\vert x \vert^{d-s}} \vert \mathscr{F}[\mu](x) \vert^2 \, dx,
		\end{equation}
	where $c(s,d)$ is the constant of \eqref{20122023matt1} for $ \alpha=s$.
	\end{lemma}
	\begin{proposition}\label{propmat}
		Let $ \gamma \in \Gamma$; then $\gamma\in\Cs$ and $ \mathrm{M}_s(\gamma)=  M_s(\gamma)$ where $\mathrm{M}_s$ is as in \eqref{defps} and $M_s$ is as in Definition \ref{defFractionalMass1currents}.
	\end{proposition}
	\begin{proof} Let $ \gamma \in \Gamma$ and $N \in \N$ as in \eqref{defGamma}. With a slight abuse of notation we identify $\gamma$ with its support, define for all $ k= 1,\dots,d$ the Radon measure $\mu_k(A):=  \int_{A\cap \gamma} \tau \cdot e_k\, d \mathcal{H}^1$ and set $\mu \coloneqq (\mu_1, \, \cdots, \mu_d)$. 
By Lemma \ref{lemmaMattila99}
		we can write
		\begin{align}
				\mathrm{M}_s(\gamma) &= \sum_{i,j=1}^N \int_{\gamma_i } \int_{\gamma_j} \frac{   \tau_i(x)  \cdot \tau_j(y)}{\vert y-x \vert^s} \, d \mathcal{H}^1(x) d \mathcal{H}^1(y)\\
				& =  \sum_{k=1}^d \int_{\gamma} \int_{\gamma} \frac{(e_k \cdot \tau(x) )(e_k \cdot \tau(y))}{\vert x-y \vert^s}\,  d \mathcal{H}^1(x) d \mathcal{H}^1(y) \\
				&=  \sum_{k=1}^d \int_{\R^d} \int_{\R^d}  \frac{1}{\vert x -y\vert^s}\, d \mu_k(x) d \mu_k(y) \nonumber\\
				&= \sum_{k=1}^d (2\pi)^{-d}c(s,d) \int_{\R^d} \frac{1}{\vert x \vert^{d-s}} \vert \mathscr{F}[\mu_k](x) \vert^2 \, dx
				\\
				&= (2\pi)^{-d}c(s,d) \int_{\R^d} \frac{1}{\vert x \vert^{d-s}} \vert \mathscr{F}[\mu](x) \vert^2 \, dx. \label{computationFracMassCsNorm}
		\end{align}
		Moreover, since $\gamma \in \Gamma$ can be identified with the 1-current $\mu$, we may observe that
		\begin{equation*}
			\mathscr{F}[\gamma](\omega) \coloneqq \overline \gamma (\mathscr{F}[\omega]) = \int_{\R^d} \langle \mathscr{F}[\omega], \mu \rangle \,dx = \int_{\R^d} \langle \omega, \mathscr{F}[\mu] \rangle \,dx 
		\end{equation*}
		for any $\omega \in \mathcal{S}(\R^d, \LamCd)$ and the representative $\widehat \gamma$ of $\mathscr{F}[\gamma]$ is equal to $\mathscr{F}[\mu]$. 
		 Finally, from \eqref{computationFracMassCsNorm}, we conclude that $\widehat{\gamma} \in \LK$ and thus the claim holds true.
	\end{proof}
	
%
%
	
	 The next lemma shows that the convolution of a 1-current $\gamma$ in $\mathcal{C}_s$ with the standard mollifiers converge to $\gamma$ in the weak star topology of $\Curr$ and that the $s$-fractional mass is continuous under this convergence. Before proving the lemma we recall that $\{\gamma_n\}_{n \in \N} \subset \Curr$ converges to $\gamma$ in the weak star topology of $\Curr$ if $\gamma_n(\omega) \to \gamma(\omega)$ as $n \to \infty$ for any $\omega \in \mathcal{D}^1$.
%
	\begin{lemma}\label{Lemmaconvolution}
		Let $ \gamma \in \mathcal{C}_s$. Then $ \gamma * \rho_{\varepsilon} \to \gamma$ as $ \varepsilon \rightarrow 0$ in the weak star topology of $\Curr$ and $  \| \gamma * \rho_{\varepsilon} \|_{\mathcal{C}_s} \rightarrow \| \gamma \|_{\mathcal{C}_s}$ as $ \varepsilon \rightarrow 0$.
	\end{lemma}
	\begin{proof}
		The proof that $ \gamma * \rho_{\varepsilon} \to \gamma$ as $ \varepsilon \rightarrow 0$ in the weak star topology of $\Curr$ is a standard result in the theory of currents. Hence we only prove that $ \| \gamma * \rho_{\varepsilon} \|_{\mathcal{C}_s}
		 \rightarrow \| \gamma \|_{\mathcal{C}_s}$ as $ \varepsilon \rightarrow 0$.
		 
		 We first check that $\gamma * \rho_{\varepsilon} \in \mathcal{C}_s$. We set $k(x) \coloneqq |x|^{-\frac{d-s}{2}}$ for $x \in \R^d \setminus \{0\}$. Since $\gamma \in \mathcal{C}_s$, we can choose the unique representative $\widehat{\gamma} \in \LK$ as in Definition \ref{setC_s}. Noticing that $k \widehat{\gamma} \in L^2(\R^d, \C^d)$, then we may compute the Fourier transform of $\gamma * \rho_{\varepsilon}$ as follows:
		 \begin{align*}
		 	\mathscr{F}[\gamma* \rho_{\varepsilon}](\omega)=  \overline{\gamma * \rho_{\varepsilon}}(\mathscr{F}[\omega]) = \overline{\gamma}(\rho_{\varepsilon} *  \mathscr{F}[\omega]) 
		 	&= \overline{\gamma}\left( \mathscr{F}\left[ \mathscr{F}[\rho_{\varepsilon}]\, \omega \right] \right) \nonumber\\
		 	&= \int_{\R^d} \langle  \mathscr{F}[\rho_{\varepsilon}]\,k^{-1}\omega(x), k\widehat{\gamma}(x) \rangle \,dx \nonumber\\
		 	&= \int_{\R^d} \langle  \mathscr{F}\left[ \rho_{\varepsilon} * \mathscr{F}^{-1}[k^{-1}\omega] \right], k\widehat{\gamma} \rangle \,dx \nonumber\\
		 	&= \int_{\R^d} \langle  \mathscr{F}^{-1}[k^{-1}\omega], \rho_{\varepsilon}*\mathscr{F}\left[k\widehat{\gamma} \right] \rangle \,dx \nonumber\\
		 	&= \int_{\R^d} \langle \omega ,  \mathscr{F}^{-1}[\rho_{\varepsilon}] \widehat{\gamma} \rangle \,dx 
		 \end{align*}
		 for any $\omega \in \mathcal{S}(\R^d, \LamCd)$. 
		 Therefore, the (unique) representative of $\mathscr{F}[\gamma * \rho_{\varepsilon}]$ is $\mathscr{F}^{-1}[\rho_{\varepsilon}]\widehat{\gamma}$. Moreover, since $ \widehat{\gamma} \in \LK$ it follows  that $\mathscr{F}^{-1}[ \rho_{\varepsilon}]\,\widehat{\gamma} \in \LK$, i.e.,   $\gamma * \rho_{\varepsilon} \in \mathcal{C}_s$. Finally, since   for any $\xi \in \R^d$ it holds true that $\mathscr{F}^{-1}[\rho_{\varepsilon}](\xi) \to 1$ as $\varepsilon \to 0$, we conclude
		 \begin{equation*}
		 	\|\mathscr{F}^{-1}[ \rho_{\varepsilon}]\,\widehat{\gamma}\|_{\LK} \to \|\widehat{\gamma}\|_{\LK} \eqqcolon \|\gamma\|_{\mathcal{C}_s} \qquad \text{ as } \e\to 0\, .
		 \end{equation*}
	\end{proof}

	\begin{proposition}[Lower semi-continuity]\label{lose}
	Let $U \subset \R^d$ be open and bounded. Let  $\{\gamma_n\}_{n \in \N} \subset \Cs(U)$  be a sequence converging to some $\gamma \in \Cs(U)$ in the weak star topology of $\mathcal{D}_1$.  Then, $M_s(\gamma) \leq \liminf_{n \rightarrow \infty}  M_s(\gamma_n)$. 
	\end{proposition}
	
	\begin{proof}
		We may assume that $\liminf_{n \rightarrow \infty}  M_s(\gamma_n) < \infty$; otherwise the claim is obvious and, up to passing to
		a further subsequence, that the liminf is actually a limit. 
		We set $k(x)= (\mathscr{F}[K_s](x))^{1/2}=|x|^{-\frac{d-s}{2}}$. Then it follows that $k \, \widehat\gamma_n \in L^2(\Rd, \C^d)$ have uniformly bounded norm and hence, up to a subsequence, 
		they weakly converge to some $\sigma  \in L^2(\Rd, \C^d)$. We set $\widetilde \sigma := k^{-1} \sigma \in \LK$; now the claim follows
		from the lower semicontinuity of the norm, once we check that $\widetilde{\sigma}$ is the representative of $\mathscr{F}[\gamma]$. To this end, let $\omega \in \mathcal{S}(\R^d, \LamCd)$. From the Fourier inversion theorem, there exists  a function $\eta \in \mathcal{S}(\R^d, \LamCd)$ such that $\omega = \overline{\mathscr{F}[\eta]}^*$. Then, recalling the notion of the extension of 1-currents as in Lemma \ref{lemmaexconconv}, we can compute $\mathscr{F}[\gamma]$ as follows:
		\begin{align*}
			\mathscr{F}[\gamma](\omega) = \mathscr{F}[\gamma](\overline{\mathscr{F}[\eta]}^*) = (2\pi)^d \, \overline{\gamma}(\overline{\eta}^*) &= \lim_{n \to +\infty} (2\pi)^d \,\overline{\gamma}_n(\overline{\eta}^*)  \nonumber\\
			&= \lim_{n \to +\infty} \mathscr{F}[\gamma_n](\overline{\mathscr{F}[\eta]}^*) \nonumber\\
			&= \lim_{n \to +\infty} \int_{\R^d} \langle k^{-1}\,\overline{\mathscr{F}[\eta]}^*, k\,\widehat{\gamma}_n \rangle \,dx \nonumber\\
			&= \int_{\R^d} \langle \overline{\mathscr{F}[\eta]}^*, \widetilde \sigma \rangle \,dx = \int_{\R^d} \langle \omega, \widetilde \sigma \rangle \,dx.
		\end{align*}
This concludes the proof  that $\widetilde{\sigma}$ is the representative of $\mathscr{F}[\gamma]$. 
		
	\end{proof}

\section{Asymptotic behavior of the $s$-fractional mass and \\ $s$-fractional curvature for smooth curves as $s \uparrow 1$}
%
%
%
%
	
	Now we show that our definition of $s$-fractional mass is consistent with the notion of classical length of a smooth curve as the parameter $s$ goes to 1. 

	\begin{proposition}\label{limitFracMassSto1SmoothCurves}
		Let $\ell>0$ and let $\gamma: [0,\,\ell] \to \mathbb{R}^d$ be an arc-length parametrization of a simple $C^{1,1}$-curve (not necessarily closed). 
		Then,
		\begin{equation*}
			\lim_{s \uparrow 1}(1-s) \, M_s(\gamma) = 2\mathcal{H}^1(\gamma) (=2\ell).
		\end{equation*}
	\end{proposition}
	\begin{proof}
		Since $\gamma$ is of class $C^{1,1}$ (and hence $\ddot{\gamma}\in L^\infty$), by the foundamental theorem of calculus we have
		\begin{equation*}
			\gamma(v) - \gamma(u) = (v-u)\,\dot{\gamma}(u) + (v-u)^2 \int_{0}^{1}(1-t)\ddot{\gamma}(u+t(v-u))\,dt 
		\end{equation*}
		and 
		\begin{equation}\label{meanValueGammaDdot}
			\dot{\gamma}(v) - \dot{\gamma}(u) = (v-u)\,\int_{0}^{1} \ddot{\gamma}(u + t(v-u)) \,dt 
		\end{equation}
		for any $u,\,v \in [0,\,\ell]$. Thus, recalling that $|\dot \gamma (u)|^2=1$ we obtain
		\begin{align}
			|\gamma(v) - \gamma(u)|^2 &= (v-u)^2 + 2(v-u)^3 \, \int_{0}^{1}(1-t)(\ddot{\gamma}(u+t(v-u)) \cdot \dot{\gamma}(u) ) \,dt \nonumber\\
			&\qquad + (v-u)^4 \, \left|\int_{0}^{1}(1-t)\ddot{\gamma}(u+t(v-u))  \,dt \right|^2 \nonumber\\
			&= (v-u)^2 \left(1 + 2(v-u) \, I_1(u,v,\gamma) + (v-u)^2\,I_2(u,v,\gamma) \right) \label{expansionGamma}
		\end{align}
		for any $u,\,v \in [0,\,\ell]$ where we defined $I_1$ and $I_2$ as
		\begin{equation*}
			I_1(u,v,\gamma) \coloneqq \int_{0}^{1}(1-t)(\ddot{\gamma}(u+t(v-u)) \cdot \dot{\gamma}(u) ) \,dt \quad \text{and} \quad I_2(u,v,\gamma) \coloneqq \left|\int_{0}^{1}(1-t)\ddot{\gamma}(u+t(v-u))  \,dt \right|^2
		\end{equation*}
		for $u,\,v \in [0,\,\ell]$.
		Recalling that $\gamma \in C^{1,1}$ is simple, there exists $c>0$ such that for any $u,\,v \in [0,\,\ell]$ it holds
		\begin{eqnarray*}
		1-\frac{1}{c}\leq 2(v-u) \, I_1(u,v,\gamma) + (v-u)^2\,I_2(u,v,\gamma) \quad\text{and}\\
			 |2(v-u) \, I_1(u,v,\gamma) + (v-u)^2\,I_2(u,v,\gamma)| <  c|u-v|.
		\end{eqnarray*}
	 Now, using \eqref{expansionGamma} and the expansion $(1+x)^{-s} = 1 - sx + O(|x|^2)$ for $x\geq-\frac{1}{c}$ where $|O(r)| \leq C|r|$ for any $r \in \R$ and some constant $C>0$ independent of $u$ and $v$, we obtain that
		\begin{align}
			&|\gamma(v) - \gamma(u)|^{-s} \nonumber\\
			&= (v-u)^{-s} \left(1 + 2(v-u) \,I_1(u,v,\gamma)  + (v-u)^2 \, I_2(u,v,\gamma) \right)^{-\frac{s}{2}} \nonumber\\
			&= |v-u|^{-s} \left( 1 - s (v-u) \, I_1(u,v,\gamma) -\frac{s}{2} (v-u)^2 \, I_2(u,v,\gamma) + O(|v-u|^2)  \right) \nonumber\\
			&= |v-u|^{-s} \left( 1 - s (v-u) \, I_1(u,v,\gamma) + O(|v-u|^2)  \right)\label{expansionGammao2}
		\end{align}     
		for any $u,\,v \in [0,\,\ell]$ with $u \neq v$. Notice that since $\gamma$ is a compact curve with bounded curvature $|\gamma(u)-\gamma(v)|$ is uniformly bounded away from zero. Thus, by recalling the definition of the $s$-fractional mass for $\gamma$ and from \eqref{expansionGamma} and \eqref{expansionGammao2}, we obtain that
		\begin{align}
	 		M_s(\gamma)
	 		&= \iint_{\gamma \times \gamma} \frac{\tau_{\gamma}(x) \cdot \tau_{\gamma}(y)}{|x-y|^s} \,d\mathcal{H}^1(x)\,d\mathcal{H}^1(y) \nonumber\\
			&=\lim_{\varepsilon \downarrow 0} \iint_{ \{|u-v| \geq \varepsilon\}} \frac{ \dot{\gamma}(u) \cdot \dot{\gamma}(v) }{|\gamma(u) - \gamma(v)|^s} \,du\,dv \nonumber\\
			&=\lim_{\varepsilon \downarrow 0}  \iint_{ \{|u-v| \geq \varepsilon\}} \frac{ 1 - \frac{1}{2}|\dot{\gamma}(u)-\dot{\gamma}(v)|^2 }{|v-u|^s} \left\{ 1 - s (v-u) \, I_1(u,v,\gamma) + o(|v-u|
			) \right\} \,du\,dv \nonumber\\
			&=   \lim_{\varepsilon \downarrow 0}\iint_{\{ |u-v| \geq \varepsilon\}} \frac{1}{|v-u|^s} \,du^,dv \nonumber\\
			&\qquad + \lim_{\varepsilon \downarrow 0}\iint_{\{ |u-v| \geq \varepsilon\}}  \left\{ - s \, \sign(v-u) |v-u|^{1-s} \, I_1(u,v,\gamma) + o(|v-u|^{1-s}) \right\} \,du\,dv \nonumber\\
			&\qquad \quad - \lim_{\varepsilon \downarrow 0}\iint_{\{ |u-v| \geq \varepsilon\}} \frac{|\dot{\gamma}(u)-\dot{\gamma}(v)|^2}{2|v-u|^s} \left\{ 1 - s (v-u) \, I_1(u,v,\gamma) + o(|u-v|) \right\} \,du\,dv \label{computationFracMassGamma}
		\end{align} 
		Notice that, from \eqref{meanValueGammaDdot} and the definitions of $I_1$ and $I_2$, the second and third terms in the right-hand side of \eqref{computationFracMassGamma} converge, and are uniformly bounded  in $s \in (0,\,1)$. 
		Therefore, taking all of the above arguments into account, we obtain
		\begin{align*}
			\lim_{s \uparrow 1}(1-s)M_s(\gamma)			&= \lim_{s \uparrow 1}\lim_{\varepsilon \downarrow 0}\iint_{\{|v-u| \geq \varepsilon\}} \frac{1}{|v-u|^s} \,du\,dv \\
			&= \lim_{s \uparrow 1}(1-s) \left\{	\int_{0}^{\ell}\int_{v}^{\ell}\frac{du}{(u-v)^s} \, dv + 	\int_{0}^{\ell}\int_{0}^{v}\frac{du}{(v-u)^s} \, dv \right\}\\
			&= \lim_{s \uparrow 1}(1-s) \, \frac{2 \ell^{2-s}}{(1-s)(2-s)} \\
			&= \lim_{s \uparrow 1}  2\ell^{2-s} = 2\ell .
		\end{align*}
	

	\end{proof}
	Let us remark that Seguin in \cite{Seguin20}  introduced a notion of $s$-fractional length for smooth curves in $\mathbb{R}^d$ and computed the corresponding limit as $s \uparrow 1$,  obtaining  that the $s$-fractional length (according with his definition) of any compact, $C^1$-curve (not necessarily closed) converges to its classical length up to proper constants (see \cite[Theorem 3.2]{Seguin20} for the detail). 
	
	In the next proposition we compute the limit as $ s \uparrow 1$ of the $s$-fractional curvature for smooth curves.
	\begin{proposition}
		Let $\gamma: [-\frac{\ell}{2}, \frac{\ell}{2}] \to \mathbb{R}^d$ be an arc-length parametrization of a simple $C^2$-curve (not necessarily closed). Then, it holds that, for all $u \in (-\frac{\ell}{2}, \frac{\ell}{2})$,
		\begin{equation}\label{limitcurv}
			\lim_{s \uparrow 1} (1-s) 	k_s(u,\gamma)= - \ddot{\gamma}(u).
		\end{equation}
		where $k_s(u,\gamma)$ is the $s$-fractional curvature of $\gamma$ at $u$ as in Definition \ref{defFractionalCurvature}.
	\end{proposition}
	\begin{proof}
		We recall that $k_s$ is defined as 
		\begin{equation}
			k_s(u,\gamma)=s\int_{-\frac{\ell}{2}}^{\frac{\ell}{2}}\frac{\dot{\gamma}(v)(\gamma(u)-\gamma(v))\cdot \dot{\gamma}(u)-\dot{\gamma}(u)\cdot \dot{\gamma}(v)(\gamma(u)-\gamma(v))}{\vert \gamma(u)-\gamma(v)\vert^{2+s}} d v.
		\end{equation}
		We prove the formula \eqref{limitcurv} for $u=0$.
		Up to a translation we have $ \gamma(0)=0$. Let $\eta>0$ be fixed, since $ \gamma \in C^2$ then for $\delta$ small enough we have for all $ v \in (-\delta,\delta)$
		\begin{equation}
			\gamma(v)=  \dot{\gamma}(0)v+ \ddot{\gamma}(0)\frac{v^2}{2}+ R_1(\gamma,v)v^2 \quad \text{ and } \quad \dot{\gamma}(v)= \dot{\gamma}(0)+ \ddot{\gamma}(0)v+ R_2(\gamma,v)v
		\end{equation}
		where $ \vert R_1(\gamma,v) \vert \leq \eta$ and $\vert R_2(\gamma,v) \vert \leq \eta$ for all $ v \in (-\delta,\delta)$.
		We have that
		\begin{equation}\label{21122023matt1}
			\begin{split}
			&-\dot{\gamma}(v)\gamma(v)\cdot \dot{\gamma}(0)+\dot{\gamma}(0)\cdot \dot{\gamma}(v)\gamma(v) \\
			&=    - (v +v^2 R_1(\gamma,v) \cdot \dot\gamma(0)) [\dot \gamma(0) + \ddot \gamma(0) v + R_2(\gamma,v) v]\\
				 & +  
				 (1+v R_2(\gamma,v) \cdot \dot \gamma(0)) [\dot \gamma(0) v  + \ddot \gamma(0) \frac{v^2}{2} + R_1(\gamma,v) v^2]  \\
				&= -\ddot{\gamma}(0)\frac{v^2}{2}+ R_3(\gamma,v)v^2
			\end{split}
		\end{equation}
%
		where $ \vert R_3(\gamma,v) \vert \leq \eta$ up to taking $\delta>0$ sufficiently small. Moreover, again for $\delta>0$ sufficiently small, we have that, for all $v \in (-\delta,\delta)$,
		\begin{equation}\label{21122023matt2}
			\gamma(v)= \gamma(0)+(\dot{\gamma}(0)+R_4(\gamma,v))v=(\dot{\gamma}(0)+R_4(\gamma,v))v
		\end{equation}
		where $ \vert R_4(\gamma,v)\vert \leq  \eta$. 
		We observe that
		\begin{equation}\label{25122023matt1}
			\lim_{s \uparrow 1}(1-s)s \int_{[-\frac{\ell}{2},\,\frac{\ell}{2}] \setminus (-\delta,\delta)}\frac{-\dot{\gamma}(v)\gamma(v)\cdot \dot{\gamma}(0)+\dot{\gamma}(0)\cdot \dot{\gamma}(v)\gamma(v)}{\vert \gamma(v)\vert^{2+s}} d v=0.
		\end{equation}

Now we fix a co-vector $\lambda \in \Lambda^1$ and assume, without loss of generality, that $\lambda(\ddot{\gamma}(0))\leq0$. Note that for $\eta>0$ small enough \eqref{21122023matt2} gives $|v|^{2+s}(1-4\eta)\leq|\gamma(v)|^{2+s}\leq |v|^{2+s}(1+4\eta)$. Hence, using \eqref{21122023matt1}, \eqref{21122023matt2}, and \eqref{25122023matt1}, we have that
		\begin{equation}\label{24122023matt0}
			\begin{split}
				&	\liminf_{s \uparrow 1}(1-s)s\int_{-\delta}^\delta \frac{ v^2 (\lambda(-\frac{\ddot{\gamma}(0)}{2}+R_3(\gamma,v)))}{\vert v \vert^{2+s}(1+4\eta)} dv \\
				&\leq \liminf_{s \uparrow 1}(1-s)s \lambda\bigg(\int_{-\frac{\ell}{2}}^{\frac{\ell}{2}}\frac{-\dot{\gamma}(v)\gamma(v)\cdot \dot{\gamma}(0)+\dot{\gamma}(0)\cdot \dot{\gamma}(v)\gamma(v)}{\vert \gamma(v)\vert^{2+s}} d v\bigg)\\
				& \leq  \limsup_{s \uparrow 1}(1-s)s \lambda\bigg(\int_{-\frac{\ell}{2}}^{\frac{\ell}{2}}\frac{-\dot{\gamma}(v)\gamma(v)\cdot \dot{\gamma}(0)+\dot{\gamma}(0)\cdot \dot{\gamma}(v)\gamma(v)}{\vert \gamma(v)\vert^{2+s}} d v\bigg)\\
				& \leq \limsup_{s \uparrow 1}(1-s)s\int_{-\delta}^\delta \frac{ v^2 (\lambda(-\frac{\ddot{\gamma}(0)}{2}+R_3(\gamma,v)))}{\vert v \vert^{2+s}(1- 4\eta)} dv.
			\end{split}
		\end{equation}
Moreover, 		we have that 
		\begin{equation}\label{24122023matt1}
			\begin{split}
				\frac{-\lambda(\ddot{\gamma}(0))+f_1(\eta)}{1+4\eta} &=\liminf_{s \uparrow 1}(1-s)s\int_{-\delta}^\delta \frac{ v^2 (\lambda(-\frac{\ddot{\gamma}(0)}{2})- \vert \lambda \vert \eta)}{\vert v \vert^{2+s}(1+4\eta)} dv\\
				&\leq	\liminf_{s \uparrow 1}(1-s)s\int_{-\delta}^\delta \frac{ v^2 (\lambda(-\frac{\ddot{\gamma}(0)}{2}+R_3(\gamma,v)))}{\vert v \vert^{2+s}(1+4\eta)} dv; \\
				&\leq \limsup_{s \uparrow 1}(1-s)s\int_{-\delta}^\delta \frac{ v^2 (\lambda(-\frac{\ddot{\gamma}(0)}{2})+R_3(\gamma,v))}{\vert v \vert^{2+s}(1-4\eta)} dv\\
				&\leq	\limsup_{s \uparrow 1}(1-s)s\int_{-\delta}^\delta \frac{ v^2 (\lambda(-\frac{\ddot{\gamma}(0)}{2}+\vert \lambda \vert \eta))}{\vert v \vert^{2+s}(1-4\eta)} dv \\
				&=\frac{-\lambda(\ddot{\gamma}(0))+f_2(\eta)}{1-4\eta}
			\end{split}
		\end{equation}
		where $ f_1(\eta), \, f_2(\eta)  \rightarrow 0$ as $ \eta \rightarrow 0^+$.
		Therefore by  \eqref{25122023matt1}, \eqref{24122023matt0}, \eqref{24122023matt1} and sending $ \eta $ to $0$ we have that
		\begin{equation}
			\begin{split}
				-\lambda(\ddot{\gamma}(0)) &\le	\liminf_{s \uparrow 1}(1-s)s \lambda\bigg(\int_{-\frac{\ell}{2}}^{\frac{\ell}{2}}\frac{-\dot{\gamma}(v)\gamma(v)\cdot \dot{\gamma}(0)+\dot{\gamma}(0)\cdot \dot{\gamma}(v)\gamma(v)}{\vert \gamma(v)\vert^{2+s}} d v\bigg)\\
				&\leq  \limsup_{s \uparrow 1}(1-s)s \lambda\bigg(\int_{-\frac{\ell}{2}}^{\frac{\ell}{2}}\frac{-\dot{\gamma}(v)\gamma(v)\cdot \dot{\gamma}(0)+\dot{\gamma}(0)\cdot \dot{\gamma}(v)\gamma(v)}{\vert \gamma(v)\vert^{2+s}} d v\bigg) = -\lambda(\ddot{\gamma}(0)).
			\end{split}
		\end{equation}
		Hence \eqref{limitcurv} follows by the arbitrariness
		of $ \lambda \in\Lambda^1$.
	\end{proof}

	\section{Compactness}
	In this section we provide a compactness result for 1-currents with finite $s$-fractional mass.
	To this end we first show that any $\gamma \in \mathcal{D}_1$ supported in $U$ can be continuously extended to a linear functional defined on $C_0^{k,\alpha}(U,\LamCd)$.
	\begin{proposition}\label{propositionEtensionCkalpha}
		Let $U\subset \R^d$ be a bounded open set, $k \in \mathbb{N}$ such that $k> d-\frac{s}{2}$ and $\alpha \in (0,\,1]$. Then, for any $\gamma \in \Cs(U)$, there exists a unique extension $\overline{\gamma}: C_0^{k,\alpha}(U,\LamCd) \to \C$ of $\gamma$ such that $\|\overline{\gamma}\|_{C^{k,\alpha}_0(U, \LamCd)^*}  \leq C\|\gamma\|_{\Cs}$ for some $C>0$, where
		\begin{equation*}
			\|\overline{\gamma}\|_{(C_0^{k,\alpha}(U,\Lambda^1))^*} \coloneqq \sup \left\{ \overline{\gamma}(\omega) \mid \, \omega \in C_0^{k,\alpha}(U, \LamCd), \, \|\omega\|_{C^{k,\alpha}}\le 1 \right\}.
		\end{equation*} 
	\end{proposition}
	\begin{proof}
		Let $ \omega \in C_c^{\infty}(U, \LamCd)$. Since, for all $|\beta|\le k$,
		$$
		\|D^\beta \omega\|_{L^1}\le |U| \|D^\beta \omega\|_{L^\infty}\le \| \omega \|_{C^k(U,\LamCd)} \vert U \vert < +\infty,
		$$ 
		then we have
		\begin{equation*}
			| \mathscr{F}[\omega] (\xi)| \le \| \omega \|_{C^k(U,\LamCd)} \vert U \vert |\xi|^{-k}
		\end{equation*}
		for any $\xi \in \R^d$. Then, recalling \eqref{17122023pom1}, we have
		\begin{align}
			\vert \gamma(\omega)\vert   &= (2\pi)^{-d}\vert   \mathscr{F}[\gamma](\overline{\mathscr{F}[\omega]}^*) \vert  \nonumber\\
			&\leq \int_{B_R(0)}\vert  \langle \overline{\mathscr{F}[\omega]}^*(\xi), \widehat\gamma (\xi) \rangle \vert \, d \xi+\int_{\R^d \setminus B_R(0)}\vert  \langle \overline{\mathscr{F}[\omega]}^*(\xi), \widehat\gamma(\xi) \rangle \vert \, d\xi 
			\\
			& \le C\,\vert U \vert\,
			\| \omega \|_{C^{k,\alpha}(U,\LamCd)} \,  \left\|\widehat\gamma |\xi|^{\frac{s-d}{2}} \right\|_{L^2} \left( \vert B_R(0) \vert  +   \left\| |\xi|^{\frac{d-s}{2} - k} \right\|_{L^2(\R^d \setminus B_R(0))} \right) \nonumber\\
			&\le
			C\| \omega \|_{C^{k,\alpha}(U,\LamCd)}  \|\gamma\|_{\Cs} < +\infty \label{uniformBounds1currentsCkalpha}
		\end{align}
		where $C$ is a positive constant independent of $\omega$. Thus, by the density of $C^{\infty}_c$ in $C^{k,\alpha}_0$, we obtain that $\gamma$ can be continuously extended into a functional on $C^{k,\alpha}_0(U, \LamCd)$ with a finite dual norm.
	\end{proof}
	For all $ \lambda>0 $, let $\Phi_\lambda \in C^{\infty}(\R^d)$ be the probability density function of a Gaussian distribution with variance equal to $\lambda$, i.e.,
	\begin{equation}\label{gausslambda}
		\Phi_\lambda(x):= \frac{ 1}{( 2 \pi \lambda^2)^{\frac{d}{2}}} \exp(-\frac{\vert x \vert^2}{2 \lambda^2}).
	\end{equation} 
%
%
%


\begin{theorem}[Compactness]\label{compthm}
	Let  $U\subset \R^d$ be a bounded open set and let $\{\gamma_n\}_{n \in \N} \subset \Cs(U)$ be such that $\|\gamma_n\|_{\Cs}\le C$ for all $ n \in \N$.
	Then, for every $k\in\N$ such that $k> d-\frac{s}{2}$ and $\alpha\in(0,1]$, we have  $\|\overline{\gamma}_n\|_{C_0^{k, \alpha}(U, \LamCd)^*} \le C$. As a consequence, there exists a subsequence $\{\overline{\gamma}_{n_i}\}_{i\in\mathbb{N}}$ such that $\overline{\gamma}_{n_i} \to \gamma $ in $C_0^{k,\alpha}(U, \LamCd)^*$ for some $\gamma \in C_0^{k,\alpha}(U, \LamCd)^*$ with $\gamma\lfloor_{\mathcal{D}^1} \in \Cs(U)$. 
\end{theorem}

\begin{proof}

From Proposition \ref{propositionEtensionCkalpha} and the assumption that $\sup_{n}\|\gamma_n\|_{C_s} \le C$, we obtain that the first part of Theorem \ref{compthm} holds true. 

To prove the second part of the theorem we  show that there exists a subsequence $\{\overline{\gamma}_{n_i}\}_{i \in \mathbb{N}}$ of $\{\overline{\gamma}_n\}_{n \in \mathbb{N}}$ which is a Cauchy sequence in $C_0^{k,\alpha}(U, \LamCd)^*$; if this holds true, then, by the completeness of $C_0^{k,\alpha}(U, \LamCd)^*$, $\{\overline{\gamma}_{n_i}\}_{i \in \mathbb{N}}$ converges to some element in $C_0^{k,\alpha}(U, \LamCd)^*$ for all $k > d-\frac{s}{2}$.
Let $ \omega \in C_0^{k,\alpha}(U, \LamCd)$ with $ \| \omega \|_{C^{k,\alpha}} \leq 1$ and let $\e>0$. Recalling the definition of  $ \Phi_\lambda$  in \eqref{gausslambda}, we have that  there exists $\lambda = \lambda_{\varepsilon} \in (0,\,1)$ such that $\Phi_{\lambda} *  \vert \cdot \vert^{\alpha}(0) < \varepsilon$. Then we have
that
\begin{align}\label{la}
	\sup_{n \in \mathbb{N}} |(\overline{\gamma}_n - \overline{\gamma}_n * \Phi_{\lambda})(\omega)| &= \sup_{n \in \mathbb{N}} |\overline{\gamma}_n (\omega - \Phi_{\lambda} * \omega)| \nonumber\\
	&\leq \sup_{n \in \mathbb{N}} \|\overline{\gamma}_n\|_{C^{k,\alpha}_0(U,\LamCd)^*} \|\omega - \Phi_{\lambda} * \omega \|_{C^{k,\alpha}(U,\LamCd)} \nonumber\\
	&\leq C \|\omega - \Phi_{\lambda} * \omega \|_{C^{k,\alpha}(U,\LamCd)} \nonumber \\
	&\leq C \| \omega \|_{C^{k,\alpha}} \int_{\R^d} \Phi_{\lambda}(x) \vert x \vert^\alpha \,  dx \nonumber\\
	&\leq  C \, \Phi_{\lambda} *  \vert \cdot \vert^{\alpha}(0) <  \e.  
\end{align}

		
We now show that the sequence $\{\overline{\gamma}_{n}\}_{n \in \N}$ is equicontinous in $\mathfrak{C}_\lambda$ where we define $\mathfrak{C}_\lambda$ by
\begin{equation}
	\mathfrak{C}_\lambda:=\left\{ \omega * \Phi_\lambda \mid \| \omega \|_{C^{k,\alpha}(U)} \leq 1 \right\}.
\end{equation}
Note that $\mathfrak{C}_\lambda$ is a compact subset of $ C^{k,\alpha}(\overline{U},\LamCd)$ thanks to the Ascoli-Arzel\`a theorem. Given any two functions $\omega,\widetilde{\omega} \in C^{k,\alpha}_0(U, \LamCd)$ with $ \| \omega \|_{C^{k,\alpha}} \leq 1$ and $\| \widetilde\omega \|_{C^{k,\alpha}} \leq 1$, we have that
\begin{align*}
	\vert \overline{\gamma}_{n}(\omega* \Phi_\lambda)- \overline{\gamma}_{n}(\widetilde{\omega}* \Phi_\lambda) \vert &\leq \| \overline{\gamma}_{n} \|_{C_0^{k}(U, \LamCd)^*} \| \omega* \Phi_\lambda-\widetilde{\omega}* \Phi_\lambda \|_{C^{k,\alpha}(U, \LamCd)}  \nonumber\\
	&\leq C  \| \omega* \Phi_\lambda-\widetilde{\omega}* \Phi_\lambda \|_{C^{k,\alpha}(U, \LamCd)},
\end{align*}
which implies the equicontinuity of $\{\overline{\gamma}_{n}\}_{n \in \N}$ in the topological space $ C\left((\mathfrak{C}_\lambda, \|\cdot \|_{C^{k,\alpha}}),\R \right)$.

Therefore, from the Ascoli-Arzel\`a theorem, we obtain that there exists a  subsequence $ \{\overline{\gamma}_{n_i}\}_{i \in \N}$, which is  a Cauchy sequence in $ C((\mathfrak{C}_\lambda, \|\cdot \|_{C^{k,\alpha}}),\R)$. In other words, we can choose $I \in \mathbb{N}$ such that, for any $i, \, j \in \mathbb{N}$ with $i,\, j\geq I$,
\begin{equation*}
	\|\overline{\gamma}_{n_i} - \overline{\gamma}_{n_j}\|_{C\left((\mathfrak{C}_{\lambda}, \|\cdot \|_{C^{k,\alpha}}),\R \right)} < \varepsilon
\end{equation*} 
where we set
\begin{equation}\label{specialNorm}
	\|f\|_{C\left((\mathfrak{C}_{\lambda}, \|\cdot \|_{C^{k,\alpha}}),\R \right)} \coloneqq \sup\left\{|f(u*\Phi_{\lambda})| \mid u \in C^{k,\alpha}_0(U,\LamCd), \, u*\Phi_{\lambda} \in \mathfrak{C}_{\lambda} \right\}
\end{equation}
for any $f \in C^{k,\alpha}_0(U,\LamCd)^*$. 

Finally, by recalling \eqref{la}, we have that, for any $i, \, j \in \mathbb{N}$ with $i ,\, j \geq I$,  
\begin{align}
	\sup_{\| \omega \|_{C^{k,\alpha}} \leq 1 } \vert (\overline{\gamma}_{n_i} - \overline{\gamma}_{n_j})(\omega)\vert &\leq \sup_{\| \omega \|_{C^{k,\alpha}} \leq 1 } \bigg(\vert (\overline{\gamma}_{n_i} - \overline{\gamma}_{n_i} * \Phi_{\lambda} )(\omega) \vert + \vert (\overline{\gamma}_{n_j} - \overline{\gamma}_{n_j} * \Phi_{\lambda} )(\omega) \vert\bigg) \\
	&\qquad + \sup_{\omega*\Phi_{\lambda} \in \mathfrak{C}_{\lambda}} \vert \overline{\gamma}_{n_i}(\omega* \Phi_\lambda)- \overline{\gamma}_{n_j}(\widetilde{\omega}* \Phi_\lambda) \vert \nonumber\\
	&< 2\varepsilon + \|\overline{\gamma}_{n_i} - \overline{\gamma}_{n_j}\|_{C\left((\mathfrak{C}_{\lambda(\varepsilon)}, \|\cdot \|_{C^{k,\alpha}}),\R\right)} \nonumber\\
	&< 3 \varepsilon.
\end{align}

Therefore,  $\overline{\gamma}_{n_i} \to \gamma $ in $C_0^{k,\alpha}(U, \LamCd)^*$ for some $\gamma \in C_0^{k,\alpha}(U, \LamCd)^*$.


We finally prove that the restriction $\gamma\lfloor_{\mathcal{D}^1}$ to $\mathcal{D}^1(U) \subset C^{k,\alpha}_0(U,\LamCd)$ belongs to $\Cs(U)$. Indeed, in the same way as in Proposition \ref{lose}, we can choose a unique function $\widetilde{\sigma} \in \LKU$ such that
\begin{equation*}
	\mathscr{F}[\gamma\lfloor_{\mathcal{D}^1}](\eta) = \int_{\R^d} \langle \eta, \widetilde\sigma \rangle \, dx
\end{equation*}
for any $\eta \in \mathcal{S}(U,\LamCd)$. Since $\supp \gamma$ is compact in $U$, so is $\supp \gamma\lfloor_{\mathcal{D}^1}$ and thus $\gamma\lfloor_{\mathcal{D}^1} \in \Cs(U)$. 

	\end{proof}


\section{Characterization of the boundary of the sets of finite $s$-fractional perimeter in $\R^2$}

Let $E \subset \R^{2}$ be a measurable set;   
we can define a 2-current $[E] \in \mathcal{D}_2$ 
associated to  $E$ by $[E]:= e_1 \wedge e_2\mathcal{L}^{2}\lfloor_{E}$. 
In the next theorem we prove that the $s$-fractional mass of the boundary of $[E]$, in the sense of currents, coincides, up to a constant prefactor, 
to the $s$-fractional perimeter of $E$. 

\begin{theorem}
	Let $ E \subset \R^{2}$ be a bounded set such that $ P_s(E)<+\infty$. Then $ \partial [E] \in \mathcal{C}_s$  and $\frac{1}{s^2}M_s(\partial [E])=P_s(E)$.
\end{theorem}
\begin{proof}
Let $R>0$ be such that $E\subset\subset B_R$ 	and let $\{E_n\}_{n \in \N}$ be a family of sets such that $ E_n \subset B_R(0)$, $ \partial E_n$ is of class $C^{\infty}$, $ \mathcal{L}^{2}(E_n \Delta E) \rightarrow 0 $ as $n \rightarrow + \infty$, and $ P_s(E_n) \rightarrow P_s(E) < +\infty$ as $n \rightarrow + \infty$; we can choose such a sequence thanks to \cite{L2018}. We set $ \gamma_n:= \partial [E_n]$. By Proposition \ref{propositionFracPerivsFracMass}, we have that $ P_s(E_n)=\frac{1}{s^2} M_s(\gamma_n)$ for all $n \in \N$ and, thus, we have that $\sup_{n\in\mathbb{N}}\|\gamma_n\|_{\Cs} < +\infty$. Hence, by Theorem \ref{compthm}, we obtain that $ \{\overline{\gamma}_n\}_{n \in \N}$ is a precompact set in $C_0^{k,\alpha}(B_R(0),\Lambda^1(\R^2) )^*$ for all $k >2-\frac{s}{2}$ where $\overline{\gamma}_n$ is the extension of $\gamma_n \in \mathcal{D}_1$ as in Proposition \ref{propositionEtensionCkalpha}. Hence, up to a subsequence, $\overline{\gamma}_n \rightarrow \gamma$ in $C_0^{k,\alpha}(B_R(0),\Lambda^1(\R^2) )^*$ with $\gamma\lfloor_{\mathcal{D}^1} \in \Cs(\overline{B_R(0)})$ for all $k >2-\frac{s}{2}$. Therefore, in particular, $\partial [E_n]=\gamma_n \to \gamma\lfloor_{\mathcal{D}^1}$ in the weak star topology of $\mathcal{D}_1(\overline{B_R(0)})$ and hence $ \partial [E]= \gamma\lfloor_{\mathcal{D}^1}$. 
	
	We now claim that $\frac{1}{s^2}M_s(\partial [E])= P_s(E)$. We first observe that, for all $\varepsilon>0$, $\partial[E] * \rho_{\varepsilon}= \partial \left( [E]* \rho_{\varepsilon}\right)$ where $\{\rho_{\varepsilon}\}_{\varepsilon>0}$ is the family of the Friedrich mollifiers. We may also observe that the representative of the Fourier transform of the boundary of the $2$-current $[E] * \rho_{\varepsilon}$ is $\mathscr{F}[\nabla^\perp (\rho_{\varepsilon} * \chi_{E})]$ where $\nabla^{\perp}u(x) \coloneqq (-\partial_2 u(x),\,\partial_1 u(x))$ for any $x \in \R^2$ and $u \in C^1(\R^2)$. Indeed, letting $\eta \in \mathcal{S}(\R^2, \Lambda^1(\C^2))$, we can choose $\omega \in \mathcal{S}(\R^2, \Lambda^1(\C^2))$ such that $\eta = \overline{\mathscr{F}[\omega]}^*$. Then, from the Plancherel-type theorem and the divergence theorem, we have
	\begin{align*}
		\mathscr{F}\left[\partial ([E] * \rho_{\varepsilon})\right]\left( \eta \right) = (2\pi)^2\, \overline{\partial \left([E] * \rho_{\varepsilon}\right)}(\omega) &= (2\pi)^2\, \int_{\R^{2}} \langle d\omega, (\rho_{\varepsilon}*\chi_{E})\,e_1\wedge  e_{2} \rangle \,dx \nonumber\\
		&= (2\pi)^2 \int_{\R^{2}} \langle \omega, \nabla^{\perp} \left(\rho_{\varepsilon}*\chi_{E}\right) \rangle \,dx \nonumber\\
		&= \int_{\R^{2}}\langle \mathscr{F}\left[ \eta \right],  \nabla^{\perp}\left(\rho_{\varepsilon}*\chi_{E}\right) \rangle \,dx \nonumber\\
		&= \int_{\R^{2}}\langle \eta,  \mathscr{F}[\nabla^{\perp}\left(\rho_{\varepsilon}*\chi_{E}\right)] \rangle \,dx.
	\end{align*} 
	Here $\overline{\partial \left([E] * \rho_{\varepsilon}\right)}$ is the extension of $\partial \left([E] * \rho_{\varepsilon}\right)$ as in Lemma \ref{lemmaexconconv}. This implies that the representative of the Fourier transform of the boundary of $ [E] * \rho_{\varepsilon}$ is $\mathscr{F}[\nabla^{\perp} (\rho_{\varepsilon} * \chi_{E})]$.

	Therefore we have

	\begin{align}
		M_s(\partial [E]* \rho_\varepsilon) = M_s(\partial ([E]* \rho_{\varepsilon})) &= (2\pi)^{-2} \int_{\R^{2}} \vert \widehat{\partial ([E]* \rho_{\varepsilon})} \vert^2(\xi)c(s,2) K_{2-s}(\vert \xi\vert) \,d\xi \nonumber\\
		&= (2\pi)^{-2}\int_{\R^{2}} \vert \mathscr{F}[\nabla^{\perp} \rho_{\varepsilon} * \chi_{E}] \vert^2(\xi) c(s,2) K_{2-s}(\vert \xi\vert) \, d\xi \nonumber\\
		&= (2\pi)^{-2}\int_{\R^{2}} \vert \xi \vert^2 \, \vert \mathscr{F}[\rho_{\varepsilon}] \vert^2 (\xi) \, \vert \mathscr{F}[\chi_{E}] \vert^2 (\xi) \frac{c(s,2)}{\vert \xi \vert^{2-s}} \,d\xi \nonumber\\
		&= (2\pi)^{-2} c(s,2)\int_{\R^{2}} \vert \xi \vert^s \, \vert \mathscr{F}[\rho_{\varepsilon}] \vert^2 (\xi) \, \vert \mathscr{F}[\chi_{E}] \vert^2(\xi) \,d\xi
	\end{align}
	where $c(s,2)$ is the constant in \eqref{20122023matt1}.
	Now sending $\varepsilon $ to $0$ by Lemma \ref{Lemmaconvolution} we obtain 
	\begin{equation}
		M_s(\partial [E])=(2\pi)^{-2} c(s,2) \int_{\R^{2}} \vert \xi \vert^s  \vert \mathscr{F}[\chi_{E}] \vert^2(\xi) d \xi= s^2P_s(E)
	\end{equation}
	where we have used \cite[Proposition 3.4]{dNVP12} in the last equality. Hence we obtain the claim.
\end{proof}
	\begin{remark}
		We observe that if $E \subset \R^2$ is a bounded measurable set such that $ \partial [E] \in \mathcal{C}_s$ then $E$ it has finite $s$-fractional perimeter. Indeed by the above proof we have 
		\begin{equation}
			M_s(\partial [E])=(2\pi)^{-2} c(s,2) \int_{\R^{2}} \vert \xi \vert^s  \vert \mathscr{F}[\chi_{E}] \vert^2(\xi) d \xi= s^2P_s(E).
		\end{equation}
	\end{remark}
\section{Approximation theorem}
	Our last result in this paper is the approximation of a divergence-free vector field with finite $s$-fractional mass by closed piecewise-linear curves. Precisely, we prove that, for any divergence-free vector field with compact support and finite $s$-fractional mass, there exists a sequence of closed piecewise-linear curves such that the corresponding  Radon measures associated with the curves converge to the vector field and its $s$-fractional mass also converge to that of the vector field. 
	
%
	
	The idea of the proof is inspired by the one by Fortuna and Garroni in \cite{FoGa23} and it goes as follows: we first decompose the support of a divergence-free vector field into finite number of small cubes. Then, in each cube, we construct a finite number of curves (lines) whose direction is almost parallel to the average of the vector field in the cube. To construct these lines, we consider the flux of the vector field on each face of each cube and, according to each flux, we choose a lattice whose size is much smaller than that of the cube. By ``properly'' choosing points on the lattice on each face, we could define a closed piecewise-constant ``curve'' (strictly speaking, some Radon measure supported on a piecewise-constant curves) such that it approximates the vector field in the sense of measure and such that its $s$-fractional mass converges to that of the vector field. Note that  the approximating curves lie in  a ``small'' neighborhood of the support of the vector field and, at this point, our strategy is different from the one by Fortuna and Garroni in \cite{FoGa23}, where the approximating curves  can stretch to infinity.  
\begin{theorem}\label{teoremaapprosi}
	Let $\psi \in C^1_c (\R^d, \R^d)$ with $\dv \psi = 0$ and let $\delta_n\to 0$.   
	
	There exists a sequence of families of closed piecewise linear curves $\Gamma_n:= \sum_{i=1}^{N(n)} \gamma^i_n$ 
	such that, denoting by $\mu_n := \sum_{i=1}^{N(n)} \delta_n \frac{\dot{\gamma}^i_n}{ |\dot{\gamma}^i_n|} \,\mathcal{H}^1\lfloor_{\gamma^i_n}$ the corresponding $1$-forms, 
	we have
	\begin{align}\label{07012024sera1}
		& \mu_n 
		\xrightarrow[n 
		\to +\infty]{} \psi \quad \text{weakly$^*$ in $\mathcal{M}_b(\mathbb{R}^d;\mathbb{R}^d)$} \\
		& |\mu_n|(\R^d)  \xrightarrow[n  \to +\infty]{} \| \psi \|_{L^1} \\
		& M_s(\mu_n) (\R^d)  \xrightarrow[n  \to +\infty]{}M_s(\psi) \, .
	\end{align}
\end{theorem}


\begin{proof}
		We divide the proof in several steps.
	
	Step 1:{\it  Construction of the curves and  error estimates in the 
	$\mathcal{M}_b(\R^d, \R^d)$ norm.}
	Fix $\rho, \, \delta, \,  \varepsilon > 0$. We cover $\supp \psi$ by a family
	$\mathcal{Q_\e} = \{Q_1, \ldots Q_{N(\varepsilon)}\}$ of pairwise disjoint  open cubes of 
	side length $\varepsilon$, with  $N(\varepsilon) \leq \frac{C}{\varepsilon^d}$.
	Let us also set $\mathcal Q_{\e,\rho}$ as the family of cubes in $ \mathcal{Q_\e}$ such that $\|\psi\|_{L^\infty}(Q)\ge \rho$. 
	sdf
	Notice that if $Q$ belongs to $\mathcal Q_{\e,\rho}$ then, for $\e$ small enough we have that  the average of $\psi$ on $Q$ is different than zero,  and we may set
	\begin{equation*}
		\bar \psi := \frac{1}{\varepsilon^d} \int_Q \psi dx, \quad \eta := \frac{\bar \psi}{|\bar \psi|}. 
	\end{equation*}
	Given  $Q\in\mathcal Q_{\e}$, consider $(d-1)$-square lattices  on the faces $F$ of $Q$ (arbitrarily oriented and) with a lattice spacing $d_F$ determined by 
	\begin{equation}\label{dedi}
		\frac{\delta \hd(F)}{d_F^{d-1}} = \left|\int_F \psi \cdot \nu_F d \hd \right| \, .
	\end{equation}
	(If $\int_F \psi \cdot \nu_F d \hd = 0$, we have the empty lattice.)
	We define
	\begin{align*}
		&S^+ := \cup \{ F \text{ face of } Q : \int_F \psi \cdot \nu_F d \hd > 0 \} \quad \text{and} \\
		&S^- := \cup \{ F \text{ face of } Q : \int_F \psi \cdot \nu_F d \hd < 0 \}.
	\end{align*}
	  Moreover, for $Q\in \mathcal Q_{\e,\rho}$ we introduce the measures $\nu^\pm:= \sum_{p\in L^\pm} \pm \delta_p$ where $L^\pm$ is the set of lattice points that lie in $S^{\pm}$.  We define $N^\pm(Q):= |\nu^\pm|(\overline Q)$.
	Consider a family $\mathcal{Z}$ of open, pairwise disjoint parallelepipeds, whose closure tiles $
	\R^d$. Assume moreover that each of the parallelepipeds in $\mathcal Z$ is
	parallel to $\eta$ and it has a square cross-section of side length $\varepsilon^2$.
	For every cylinder $Z\in{\mathcal Z}$ we define $N^\pm(Z) = |\nu^\pm|(Z)$  and $L(Z) := |N^+(Z) - N^-(Z)|$.
	
	Now, we estimate the following quantities:
	\begin{itemize}
		\item $|N^+(Q) - N^-(Q)|$:
		By \eqref{dedi} we have 
		\begin{equation*}
			d_F \geq (\|\psi\|^{-1}_{L^\infty}  \delta)^{\frac{1}{d-1}} =: \Delta.
		\end{equation*}
		Notice that, for each face $F$ of $Q$,  $\delta$ times the number of lattice points $|\nu|(F)$ in the face $F$ coincides with the flux of $\psi$ through $F$,  up to  errors arising from  discretization; 
		more precisely, 
		\begin{equation*}
			\left| |\nu|(F) - \frac{1}{\delta} \left| \int_F \psi \cdot \nu_F d \hd \right| \right|
			\leq C \left( \frac{\varepsilon}{\Delta} \right)^{d-2}.
		\end{equation*}
		Since $\psi$ is divergence-free, the flux through $\partial Q$ is zero and it follows that
		\begin{equation*}
			|N^+(Q) - N^-(Q)| \leq C \left( \frac{\varepsilon}{\Delta} \right)^{d-2}
			=C \left( \varepsilon\delta^{\frac{1}{1-d}} \right)^{d-2}.
		\end{equation*}
		
		\item L(Z): 
				We denote by $\tilde \psi:\partial Q\to \R^d$ the piecewise constant function 
				coinciding with  $\frac{1}{|F|} \int_F \psi  \hd$ on each face $F$ of $Q$.  
				Let 
				$$
				A^+:= \{ x\in \partial Q\cap Z : \bar \psi \cdot \nu_{\partial Q}(x) \ge 0 \}, \quad 
				A^-:= \{ x\in \partial Q\cap Z : \bar \psi \cdot \nu_{\partial Q}(x) < 0 \} \, 
				$$ 
				 	As in the step above, we have 
	\begin{equation}
			\left |    |N^+(Z) - N^-(Z)| - \frac{1}{\delta}   
			\left| \int_{(A^+ \cup A^-)} \tilde \psi \cdot \nu d \hd\right| \right|\le C \left( \frac{\varepsilon^2}{\Delta} \right)^{d-2} \, .
					\end{equation}

		Notice  that for the constant function $\bar \psi$ (which is parallel to Z), we have 
		\begin{equation*}
			\int_{A^+}\bar \psi \cdot \nu  d \hd  + 
			\int_{A^-}\bar \psi \cdot \nu  d \hd =0 \, .
		\end{equation*}
		Therefore, using that $\|\tilde\psi-\bar\psi\|_{L^{\infty}(Q)}\leq C\, \e$ we obtain
		\begin{align*}
			L(Z) &= |N^+(Z) - N^-(Z)| \\
			& \leq C \left( \frac{\varepsilon^2}{\Delta} \right)^{d-2}  + \left| \frac{1}{ \delta } 
			\int_{A^+} \tilde \psi \cdot \nu d \hd + \frac{1}{ \delta} 
			\int_{A^-} \tilde \psi \cdot \nu d \hd  \right| \\
			&=  C \left( \frac{\varepsilon^2}{\Delta} \right)^{d-2}  + \left| \frac{1}{ \delta } 
			\int_{A^+} (\tilde \psi-\bar\psi) \cdot \nu d \hd + \frac{1}{ \delta} 
			\int_{A^-} (\tilde \psi-\bar\psi) \cdot \nu d \hd  \right| \\
			&\leq C \left( \frac{\varepsilon^2}{\Delta} \right)^{d-2} + C\frac{\varepsilon \hd (Z \cap S^+)}{\delta} +  C\frac{\varepsilon \hd (Z \cap S^-)}{\delta}\\ 
			& \le  C \left( \frac{\varepsilon^2}{\Delta} \right)^{d-2} + C \frac{\e^{2d-1}}{\delta} \le C \frac{\e^{2d-4}}{\delta^{\frac{d-2}{d-1}}} + C \frac{\e^{2d-1}}{\delta}. 
		\end{align*}
		Since the cardinality of all the cylinders $Z$ intersecting $Q$ is estimated by $C\,\e^{1-d}$ we obtain 
		\begin{equation}\label{ellezeta}
			\sum_{Z: \, Z\cap Q \neq \emptyset} L(Z) 
			\leq C \left(  \frac{\varepsilon^{d-3}}{\delta^{\frac{d-2}{d-1}}}  + \frac{\varepsilon^d}{\delta}\right) \, .
		\end{equation}
			\end{itemize}
	
	Now   we have to connect as much as possible pairs of points with positive and negative signs  inside $Q$; we describe such a construction step by step. 
	If $Q\in\mathcal Q_{\e,\rho}$, we first connect all such pairs that lie in the same cylinder $Z$ and we repeat  this process for all the
		cylinders. To this purpose, fix a cylinder $Z$ and set $N_{min}(Z):= \min \{ N^-(Z), \,  N^+(Z)\}$. Then, arbitrarily select  $N_{min}(Z)$ positive and negative masses 
		$\{x^\pm_1, \ldots, x^\pm_{N_{min}(Z)}\} \subset \supp(\nu^{\pm} \llcorner Z)$ 
		and connect them with segments, obtaining a  one current $\mu_Z$ such that $\partial \mu_Z = \sum_{i=1}^{N_{min}(Z)} \delta_{x^+_i} - \delta_{x^-_i}$.  Summing over all cylinders $Z$ intersecting $Q$ we get a current denoted by $\mu^g_Q$.

Notice that for any tangent unit vector $\frac{x_i^+ - x_i^-}{|x_i^+ - x_i^-|}$ to  $\mu^g_Q$ either $|x_i^+ - x_i^-| \le \rho \e$ or $\frac{x_i^+ - x_i^-}{|x_i^+ - x_i^-|}$ differs from $\eta$ at most by $C(\rho) \e$, for some constant $C(\rho)$ depending on $\rho$.
Therefore, 
by \eqref{ellezeta}, one can easily prove that
		\begin{equation}\label{ellezeta2}
			||\psi \mathcal{L}^d|(Q)  - \delta |\mu^g_{\varepsilon,\delta}|(Q)| \le  C(\rho)  \left(   {\delta^{\frac{1}{d-1}   }}  \varepsilon^{ d-2}  + \e^{d+1}   \right) + C \rho \e^d \, .
		\end{equation}
Analogously, one can easily prove that for all $ \omega \in C_c^{0}(Q ; \Lambda^1)$ with $ \| \omega \|_{C_c^{0}(Q ; \Lambda^1)} \leq 1$ 
		\begin{equation}\label{ellezeta3}
			| \langle \psi   - \delta \mu^g_{\varepsilon,\delta} , \omega \rangle \vert \le  C(\rho)  \left(   {\delta^{\frac{1}{d-1}   }}  \varepsilon^{ d-2}  + \e^{d+1}  \right)  + C \rho \e^d \, .
		\end{equation}

Now,  either if $Q\in \mathcal Q_\e$ or $Q\in \mathcal Q_{\e,\rho}$,  on $\partial Q$ there are still many positive and negative masses, and as before we connect as much as possible of these masses with segments, obtaining a 1-current $\mu^b_Q$. By construction we have
$$
|\nu^+ + \nu^- - \partial (\mu^g_Q + \mu^b_Q)|= |N^+(Q) - N^-(Q)| \, .
$$
Note that this construction does not connect all Dirac masses we have constructed on $\partial Q$. Some of them will be actually left isolated and will not contribute to the approximation procedure. We set 	
$$
\mu^b_{\e,\delta}:= \sum_Q \mu^b_Q, \qquad 
\mu^g_{\e,\delta}:= \sum_Q \mu^g_Q, \qquad
\mu_{\e,\delta}^{Tot}:= \mu^b_{\e,\delta} +\mu^g_{\e,\delta} \, .
$$
Notice that, in general, $\mu_\e^{Tot}$ is not closed. 
Anyway, by construction $\partial \mu_\e^{Tot}(\R^d)=0$. Therefore, we can close $\mu_\e^{Tot}$ connecting its boundary (which contains the same number of positive and negative masses) with suitably oriented segments, namely,  by adding a $1$-current  $\mu_{\varepsilon,\delta}^l$. Recalling that $\psi$ has bounded support,
the number of such oriented segments, and therefore the total mass of  $\delta \mu_{\varepsilon,\delta}^l$ can be bounded by
	\begin{equation}\label{muelle}
		|\delta \mu_{\varepsilon,\delta}^l| (\R^d) \le  C 	\delta \sum_{Q \in \mathcal {Q}} |N^+(Q) - N^-(Q)| \leq 
		C \varepsilon^{-2}  \delta^{\frac{1}{d-1}}\, .
	\end{equation}


Recalling \eqref{ellezeta} we have

\begin{eqnarray}\label{varbad}	
 |\delta \mu_{\varepsilon,\delta}^b|(\R^d) &\le& \e\delta \sum_{Q}
			\sum_{Z: \, Z\cap Q \neq \emptyset} L(Z) + C \rho\\
			&\leq&  C \delta \e^{-d+1} \left(  \frac{\varepsilon^{d-3}}{\delta^{\frac{d-2}{d-1}}}  + \frac{\varepsilon^d}{\delta}\right) + C \rho
			= C   \left(  \varepsilon^{-2}\delta^{\frac{1}{d-1}}  + {\e}\right)  + C \rho \, .
\end{eqnarray}

Now, we set
$$
\mu_{\e,\delta}:= \delta ( \mu^b_{\e,\delta} +\mu^g_{\e,\delta} + \mu_{\varepsilon,\delta}^l )\, .
$$
By construction $\partial \mu_{\e,\delta}=0$. Moreover, by \eqref{muelle} and \eqref{varbad} we have

$$
|\mu_{\e, \delta} -\delta \mu^g_{\e,\delta}|(\R^d) \le r(\rho, \e,\delta)  \qquad \text{ with } \limsup_{\rho \to 0} \limsup_{\e \to 0} \limsup_{\delta \to 0} r(\rho, \e,\delta)=0\, .
$$

\noindent Step 2: {\it Strict convergence to $\psi$.}

		By \eqref{ellezeta2} we have
	\begin{eqnarray*}
		||\psi \mathcal{L}^d| (\R^d) - |\mu_{\varepsilon,\delta}|(\R^d)|  &\leq &
		\sum_{Q \in \mathcal{Q_{\e,\rho}}} ||\psi \mathcal{L}^d|(Q)  - \delta |\mu^g_{\varepsilon,\delta}|(Q)|  + r(\rho, \e,\delta)\\
		&\le& C \e^{-d} C  \left(   {\delta^{\frac{1}{d-1}   }}  \varepsilon^{ d-2}  + \e^{d+1}  \right)  + r(\rho, \e,\delta) \le r(\rho,\e,\delta) \, .
 \end{eqnarray*}
Analogously, by \eqref{ellezeta3} it easily follows  that,  for all $ \omega \in C_c^{0}(\R^d, \LamRd)$ with $ \| \omega \|_{C_c^{0}(\R^d, \LamRd)} \leq 1$ 
		$$
			| \langle \psi   - \delta \mu^g_{\varepsilon,\delta} , \omega \rangle | \le  r(\rho,\e,\delta)\, .
		$$

\noindent Step 3: {\it Estimate of the error for the $s$-fractional mass and conclusion of the proof.}
By the steps $1$ and $2$, using a standard diagonal argument, we deduce that there exists a sequence $\{\mu_{n}\}_{n \in \N}$  satisfying the first two formulas in  \eqref{07012024sera1}.
Moreover, by construction there exists a constant $C$, depending only on $\|\psi \|_{L^\infty}$ such that 
$|\mu_{n}|(B_r(x)) \le Cr$ for all $x\in\R^d$, \, $r>0$.  Let  $\nu_n$ be the measure on $\R^+$ defined as the derivative of the (monotone) function  $r\mapsto  |\mu_{n}|(B_r(x))$. 
Notice that the average of $\nu_n$ on every right neighborhood of zero is less then or equal to $C$; by an easy rearrangement argument, for all $x\in\R^d$ and $R>0$ we have
$$
\int_{B_R(x)} \frac{1}{|x-y|^s} \, d|\mu_{n}| = \int_0^R  \frac{1}{t^s} \, d \nu_{n} \le  C \int_0^R \frac{1}{t^s} \, dt = CR^{1-s} \, .
$$
Then, setting $D(R):= \{(x,y) \in \R^d \times \R^d \mid |x -y | < R \}$ we have
$$
\int_{\R^d} \int_{\R^d} \chi_{D(R)} \frac{1}{|x-y|^s} d|\mu_{n}| \otimes d|\mu_{n}| \le C R^{1-s},
$$	
and the same estimate clearly holds true replacing $\mu_n$ with $\psi \, d\mathcal L^d$.
Finally, in view of the first two formulas in  \eqref{07012024sera1} it easily follows that
$$
\int_{\R^d} \int_{\R^d} (1-\chi_{D(R)}) \frac{1}{|x-y|^s}  \frac{d \mu_n}{d |\mu_n|} \cdot  \frac{d \mu_n}{d |\mu_n|}   d|\mu_{n}| \otimes d |\mu_{n}|  
$$	
converges, as $n\to +\infty$, to 
$$
\int_{\R^d} \int_{\R^d} (1-\chi_{D(R)}) \frac{1}{|x-y|^s}  \psi(x) \cdot \psi(y) \, dx \, dy.
$$	
Letting $R\to 0$ we get also the third formula in \eqref{07012024sera1}  .
\end{proof}

We recall the following classical lemma in the theory of the currents (see for instance  \cite[Lemma 7.3.3 page 186]{krantz2008geometric}).
\begin{lemma}\label{Lconvcorrenti}
	Let $ \gamma \in \mathcal{D}_1$ with compact support. For all $ \varepsilon>0$ there exists a unique function $f_\varepsilon \in C_c^{\infty}(\R^d,\R^d)$ such that 
	 $ \gamma * \rho_\varepsilon (\omega)= \int_{\R^d} \langle \omega(x), f_\varepsilon (x) \rangle d x$ for all $ \omega \in C_c^{\infty}(\R^d,\LamRd)$.
\end{lemma}
\begin{corollary}\label{corollaryApproxCurrents}
	Let $\gamma \in \mathcal{C}_s$ without boundary.  Then, there exists a sequence of families of closed piecewise linear curves $\Gamma_n:= \sum_{i=1}^{N(n)} \gamma_{n}^i$ such that, denoting by $\mu_n := \sum_{i=1}^{N(n)} \delta_n \frac{\dot{\gamma}^i_n}{ |\dot{\gamma}^i_n|} \,\mathcal{H}^1\lfloor_{\gamma^i_n}$ the corresponding $1$-forms, 
	we have
	\begin{align}
		& \mu_n 
		\xrightarrow[n 
		\to +\infty]{} \gamma \quad \text{weakly$^*$ in $\mathcal{D}_1$} \\
		& M_s(\mu_n) (\R^d)  \xrightarrow[n  \to +\infty]{}M_s(\gamma) \, .
	\end{align}
\end{corollary}	
	\begin{proof}
		By the above lemma we have that for all $\varepsilon>0$ there exists a unique function $f_\varepsilon:= (f_\varepsilon^1, \dots, f_\varepsilon^d) \in C_c^{\infty}(\R^d,\R^d)$  such that $ \gamma * \rho_\varepsilon (\omega)= \int_{\R^d} \langle \omega(x), f_\varepsilon (x) \rangle d x$ for all $ \omega \in C_c^{\infty}(\R^d,\LamRd)$. Moreover we have that  $ \div f_\varepsilon =0$. Indeed, from the closedness of $\gamma$ and the divergence theorem, we have that 
		\begin{equation*}
			0= \partial \gamma(\rho_{\varepsilon} * g) =  \gamma(d(\rho_{\varepsilon}*g))=\gamma(\rho_{\varepsilon}*d g)=\int_{\R^d} \langle d g(x), f_\varepsilon(x)\rangle dx= \int_{\R^d} g(x) \div f_\varepsilon(x) dx
		\end{equation*}
 		for any $g \in \mathcal{D}_0(\R^d)$. The thesis follows using a standard diagonal argument combining the Lemma \ref{Lemmaconvolution} applied to $\gamma$ and Theorem \ref{teoremaapprosi} applied to $ (f_\varepsilon^1,\dots,f_\varepsilon^d)$. 
		\end{proof}

\section*{Conclusions}
In this paper we have proposed a notion of $s$-fractional mass for $1$-currents in $\R^d$, where $s$ belongs to $(0,1)$. It is natural to ask what happens for the limit cases corresponding to $s=0$ and $s=1$. We have done a partial progress in this program, showing that, for regular enough curves, the limit as $s\to 1$ gives back the classical notion of length. This pointwise result could be investigated in terms of $\Gamma$-convergence, as done for $s$-fractional perimeters in \cite{Da2002,ADPM11}.   

Moreover, in the Introduction we have suggested the following notion of $\e$-regularized mass for $s$ equal to $1$: 	
\begin{equation}\label{deforga2conc}
	\mathrm{M}_1^\varepsilon (\gamma) = \int_{\gamma}\int_{\gamma} \frac{ \tau(x)\cdot \tau(y)}{\max\{|x-y|, \varepsilon\}}  \, d \H^1\, d\H^1\ \,.
\end{equation}
We expect that such an energy functional (in its material dependent declinations) could represent a purely geometric counterpart to Ginzburg Landau energy functionals for vortex filaments, as well as to the self energy of dislocation lines in crystals. In this respect, it seems interesting to study the $\Gamma$-convergence of these functionals, after scaling them by $\frac{1}{|\log \e|}$,  to line tension models, as $\e \to 0$ \cite{GM06, DKP22}. We also remark that, for technical reasons, it could be convenient to replace the denominator ${\max\{|x-y|, \varepsilon\}}$ with a function which is strictly  increasing also around zero. 
The study of the limit as $s\to 0$ of the $s$-fractional mass could also deserve some attention.  

As explained in the Introduction, the proposed notion of $s$-fractional mass generalizes that of $s$-fractional perimeters of sets in $\R^2$ to higher codimention and not necessarily integer rectifiable currents. 
In fact, as for sets with finite fractional perimeter, there is no reason to expects that the $1$-currents with bounded fractional mass are rectifiable; now, a natural question arises: what are integer (not necessarily rectifiable) currents in this context? One possibility is to exploit our Smirnov type approximation result, defining the integer currents with finite mass as the closure of polyhedral $1$-currents with respect to the notion of convergence in Theorem  \ref{teoremaapprosi}. Now, take an integer (i.e., limit of smooth curves) $1$-current $\gamma$ without boundary and with finite $s$-fractional mass; is there a $2$-current, with bounded area (i.e., bounded local mass) whose boundary is $\gamma$? Furthermore, is the mass controlled by the $s$-fractional mass of its boundary? In the terminology of geometric measure theory, we ask if the fractional mass controls the flat norm. In the flat case of 1-currents in $\R^2$ this is true, being  nothing but the fractional isoperimetric inequality.  
A positive answer to the questions posed above would suggest that integer $1$-currents with bounded fractional mass are a good class of boundary data for setting up the Plateau problems for minimal surfaces. 

Let us also mention that our definition of higher codimensional fractional mass could be generalized to the case of  $k$-forms in $\R^d$ for all $1\le k\le d$. Finally, also the case $k=0$ corresponding to charged masses could be investigated.

\subsection*{Acknowledgments}  
	A. Kubin is supported by the DFG Collaborative Research Center TRR 109 “Discretization in Geometry and Dynamics”. 
\bibliographystyle{plain}
\bibliography{biblio_fractional_mass22}
	
%
\end{document}